\documentclass[reqno, a4paper]{amsart}
\usepackage[english]{babel}
\usepackage[utf8]{inputenc}

\usepackage{color,bm}
\usepackage{enumerate}
\usepackage{amsthm,amssymb}
\usepackage{amsmath}
\usepackage{graphicx,epsfig}
\usepackage{hyperref}
\usepackage{amsfonts, mathtools, subfig, makebox, dsfont}
\usepackage{comment}

\mathtoolsset{showonlyrefs}
\hypersetup{hidelinks}

\allowdisplaybreaks[1]

\def\C{{\mathbb{C}}}

\def\R{{\mathbb{R}}}

\def\N{\mathbb N}

\newtheorem{theo}{Theorem}[section]
\newtheorem{lemma}[theo]{Lemma}
\newtheorem{prop}[theo]{Proposition}
\newtheorem{cor}[theo]{Corollary}

\newtheorem{hyp}[theo]{Hypothesis}
\theoremstyle{definition}
\newtheorem{exa}[theo]{Example}
\newtheorem{rem}[theo]{Remark}

\def\f{\bm{f}}
\def\uu{\bm{u}}

\def\calD{{\mathcal{D}}}

\def\H{\mathcal H}

\def\eps{\varepsilon}

\newcommand{\D}{\nabla}

\newcommand{\norm}[1]{\left\Vert#1\right\Vert}
\newcommand{\abs}[1]{\left\vert#1\right\vert}

\newcommand{\sign}[1]{\operatorname{sign}\left(#1\right)}

\DeclareMathOperator*{\esssup}{ess\,sup}

\let\div\undefined
\DeclareMathOperator{\div}{\mathrm{div}}

\begin{document}
\numberwithin{equation}{section}
\title[$L^p$-maximal regularity for vector-valued Schr\"{o}dinger operators]{$L^p$ Maximal regularity for vector-valued Schr\"{o}dinger operators}

\author[D. Addona, V. Leone, L. Lorenzi, A. Rhandi]{Davide Addona, Vincenzo Leone, Luca Lorenzi \& Abdelaziz Rhandi}

\address{D.Addona, L.Lorenzi: Plesso di Matematica, Dipartimento di Scienze Matematiche, Fisiche e Informatiche, Università di Parma, Parco Area delle Scienze 53/A, 43124 Parma, Italy}
\email{davide.addona@unipr.it, luca.lorenzi@unipr.it}

\address{V.Leone, A.Rhandi: Dipartimento di Matematica, Università degli Studi di Salerno, Via Giovanni Paolo II, 132, 84084 Fisciano (SA), Italy}
\email{vleone@unisa.it, arhandi@unisa.it}

\thanks{This article is based upon work from COST Action CA18232 MAT-DYN-NET, supported by COST (European Cooperation in Science and Technology) and the project PRIN2022 D53D23005580006 ``Elliptic and parabolic problems, heat kernel estimates and spectral theory". The authors are also members of G.N.A.M.P.A. of the Italian Istituto Nazionale di Alta Matematica (INdAM). The authors have been partially funded by the project INdAM - G.N.A.M.P.A. ``Operatori ellittici vettoriali a coefficienti illimitati in spazi $L^p$''.}
\vskip 0.5cm

	
\begin{abstract}
In this paper we consider the vector-valued Schr\"{o}dinger operator $-\Delta + V$, where the potential term $V$ is a matrix-valued function whose entries belong to $L^1_{\rm loc}(\R^d)$ and, for every $x\in\R^d$, $V(x)$ is a symmetric and nonnegative definite matrix, with non positive off-diagonal terms and with eigenvalues comparable each other. For this class of potential terms we obtain maximal inequality in $L^1(\R^d,\R^m).$ Assuming further that the minimal eigenvalue of $V$ belongs to some reverse H\"older class of order $q\in(1,\infty)\cup\{\infty\}$, we obtain maximal inequality in $L^p(\R^d,\R^m)$, for $p$ in between $1$ and some $q$. 
\end{abstract}
	
\maketitle

\section{Introduction}
The aim of this paper is to establish the vector-valued apriori maximal $L^p$-inequalities
\begin{eqnarray*}
\|\Delta \uu \|_{L^p(\R^d,\R^m)}+\|V \uu \|_{L^p(\R^d,\R^m)} \le \|\Delta \uu -V\uu \|_{L^p(\R^d,\R^m)},\qquad\;\, \uu \in C_c^\infty(\R^d,\R^m),
\end{eqnarray*}
for a class of symmetric nonnegative definite matrix-valued potentials $V$, with negative off-diagonal terms and with eigenvalues which are comparable
each other. This class includes matrix potentials with polynomially growing and singular entries, see Examples \ref{exemple1} and \ref{exemple2}.
As a consequence, one obtains that the $L^p$-realization of the operator $\Delta -V$, with domain
\begin{eqnarray*}
W^{2,p}(\R^d,\R^m)\cap \{\uu\in L^p(\R^d,\R^m): V\uu \in L^p(\R^d,\R^m)\},
\end{eqnarray*}
generates a $C_0$-semigroup of contractions on $L^p(\R^d,\R^m)$, which is also analytic if $p>1$.
This generalizes the results in \cite{auscher-benali:2007} and \cite{Shen} to the vector-valued case.

By using a perturbation method, due to Monniaux and Pr\"uss, see \cite{MP97}, these type of maximal inequalities have been obtained in the seminal paper \cite{HLPRS} and more recently in \cite{KMR18}, where a class of vector-valued Schr\"odinger type operators is considered, when the matrix potential $V$ is quasi
accretive and locally Lipschitz continuous on $\R^d$, and $D_jV(-V)^{-\alpha} \in L^\infty(\R^d,\R^{m\times m})$
for some $\alpha \in [0,1/2)$. This condition excludes the case of matrix-valued potentials with singular entries and also those growing more than quadratically. 

We quote the papers
\cite{AngLorMan,kunze-maichine-rhandi:2018,MR18}, where the 
operator $\div(Q\nabla) -V$ on $L^p(\R^d,\R^m)$ is considered 
when the matrix-valued diffusion function $Q$ is bounded and can degenerate neither at some $x \in \R^d$ nor at infinity, and generation of a semigroup in the $L^p$-spaces, with the description of the domain of its generator, is proved. More recently, 
in \cite{ALMR1}, more general diffusion and potential matrices are considered, under assumptions which still exclude potentials that have singularities at some $x\in \R^d$.

A class of vector-valued elliptic operators, including also first- and second-order coupling term, has been considered very recently in  \cite{AngLorMan-1,ALMR,AngLorPal}.

In the $L^p$-context, in the case of elliptic operators with unbounded diffusion and drift coefficients, it is known, at least in the scalar case, that the suitable space to study generation and regularity properties of semigroups generated by elliptic operators with unbounded coefficients, are $L^p$-spaces related to the invariant measure associated to such operators. In the vector-valued case, only partial results in this direction are available so far, see \cite{AAL_Inv,AAL_Inv1,AL,AngLor20}.

\subsection*{Organization of the paper} In Section \ref{sect-FFP}, we begin by proving a vector-valued version of a Fefferman-Phong type inequality. This allows us, in Section \ref{sect-defH}, to study the homogeneous version of $\mathcal{H}$, the $L^2$-version of the operator $-\Delta +V$. In Section \ref{sec:L1}, we prove the $L^1$-maximal inequalities. Here, the main ingredient is suitable approximations of the matrix-potential $V$ that preserve the positivity (componentwise) of the resolvent of $-\mathcal{H}$, see Proposition \ref{prop:conv_u_eps_M}. To prove the $L^p$-maximal inequalities we first prove, in Section \ref{sec:Lp}, a vector-valued version of \cite[Theorem 3.14]{aus-mar} that permits us, together with the $L^1$-maximal inequalities, to obtain the $L^p$ estimates. In Section \ref{sec:Gen} we adapt the celebrated perturbation theorem by T. Kato, see \cite[Theorem 3]{kato86}, to the vector-valued case and deduce that $L^p$-realization of the operator $\Delta -V$ with domain 
\begin{eqnarray*} 
W^{2,p}(\R^d,\R^m)\cap \{\uu\in L^p(\R^d,\R^m): V\uu \in L^p(\R^d,\R^m)\}
\end{eqnarray*}
generates a $C_0$-semigroup of contractions on $L^p(\R^d,\R^m)$ which is, for $p>1$, also analytic.
Section \ref{sec:Exa} is dedicated to two examples of applications.

\subsection*{Notation}
The absolute value of a vector $\xi\in\R^j$ $(j\in\N)$ is the vector $|\xi| = (|\xi_1|,\dots,|\xi_j|)\in\R^j.$
For $ d,\,m,\,k\in\N,$ $C^{k}(\R^d,\R^m)$ is the space of all vector-valued functions $\f:\R^d\to\R^m$, which are continuously differentiable up to the $k$-th order in $\R^d$. $C_c^{\infty}(\R^d,\R^m)$ is the space of the compactly supported and infinitely differentiable functions $\f:\R^d\to\R^m$ and, for $p\in[1,\infty)\cup\{\infty\},\ L^p(\R^d,\R^m) $ denotes the space of (the classes of) all measurable vector-valued functions, such that
$\norm{\f}_p= \left(\int_{\R^d}\norm{{\f}(x)}^pdx\right)^{\frac1p}$
 is finite, if $1 \leq p <\infty$, and such that
$\norm{\f}_{\infty}= \esssup_{x\in\R^d}\norm{{\f}(x)}<\infty$, if $p=\infty.$ In particular, we denote with $L^p_c(\R^d,\R^m)$ the subspace of $L^p(\R^d,\R^m)$ of compactly supported functions on $\R^d$ and with $L^p_{\rm loc}(\R^d,\R^m)$ the set of the measurable functions that belong $L^p(\Omega,\R^m)$ for every bounded measurable subset $\Omega$ of $\R^d$.

$W^{k,p}(\R^d,\R^m)$ denotes the Sobolev space of order $k$ in $L^p(\R^d,\R^m)$, that is the space of functions $\f\in L^p(\R^d,\R^m)$ such that the distributional derivative $\partial^{\beta}\f$ belongs to $L^p(\R^d,\R^m)$ for any multi-index $\beta$ with length at most  $k$. When $p=2$, we write $H^k(\R^d,\R^m)$ instead of $W^{k,2}(\R^d,\R^m)$. The space $W_{\rm loc}^{k,p}(\R^d,\R^m)$ consists of those measurable and locally integrable functions which, along with their distributional derivatives up to order $k$, belong to $L^p_{\rm loc}(\R^d,\R^m)$.

Let $ V=(v_{ij}),\ W=(w_{ij})$ be real $ m\times m $ matrices. We say that $ V\le W$ componentwise if $ v_{ij}\le w_{ij} $ for all  $ i,\,j=1,\dots,m,$ while $ V\le W$ in the sense of forms if $ \langle V\xi,\xi\rangle\le\langle W\xi,\xi\rangle $ for all $\xi\in\R^m.$
The average of a function $f$ over a subset $X$ of $\R^d$ will be denoted as $\operatorname{av}_{X}(f) = \frac1{|X|}\int_X f(y)\,dy$, where $|X|$ denotes the Lebesgue measure of $X$, which is assumed to be finite. 
With $\chi_E $ we denote the characteristic function of the set $E.$

A scalar function $w$ belongs to the reverse 
H\"{o}lder class $B_q,$ for $ q\in(1,\infty)\cup\{\infty\}$, if it is almost everywhere positive, $w\in L_{\rm loc}^q(\R^d)$ and there exists $C>0$ such that for all $Q$ cubes of $\R^d$ we have
\begin{align*}
\left(\frac1{|Q|}\int_Q w^q(x)\,dx\right)^{\frac1q}\le C\frac1{|Q|}\int_Qw(x)\,dx.
\end{align*}
For $q=\infty$, the left hand side in the last condition is replaced by $\esssup_{x\in Q}w(x)$.

If $K$ is a Hilbert space and $K'$ is its topological dual, then we denote by $\langle k',k\rangle_{K',K}$ the duality in $K$, for every $k\in K$ and $k'\in K'$.

\section{Main assumptions}
In the whole manuscript we assume the following hypothesis.
\begin{hyp}
\label{hyp-1}
Let $ V\colon \R^d\to\R^{m\times m}  $	be a matrix-valued operator such that $ v_{ij} = v_{ji}\in L^1_{\rm loc}(\R^d) $ for all $ i,j\ \in\{1,\dots,m\} $, the off-diagonal terms are non-positive, i.e., $ v_{ij}\le0 $ for all $ i\ne j\in\{1,\dots,m\}$, and for almost every $x\in\R^d$
\begin{align*}
\langle V(x)\xi, \xi \rangle \ge 0,\qquad\;\,\xi\in\R^m.
\end{align*}
\end{hyp}

Let $ \lambda_V(\cdot)\coloneqq\min_{\uu\in\R^m\colon\|\uu\|=1}\langle V(\cdot)\uu,\uu\rangle $ be the minimal eigenvalue of $V$. It is an almost everywhere in $\R^d$ non-negative and locally integrable scalar function. In addition we assume the following assumption.
\begin{hyp}
\label{hyp-2}
The eigenvalues of $V$ are comparable each other, i.e., if we denote with $\Lambda_V$ the maximal eigenvalue of $V$, then there exists a constant $C\geq 1$ such that for almost every $x\in\R^d$
\begin{equation*}
\lambda_V(x)\le\Lambda_V(x)\le C\,\lambda_V(x).
\end{equation*}
\end{hyp}

\section{A Fefferman-Phong type inequality}
\label{sect-FFP}
We start with an inequality which will be fundamental for the definition of our operator.
\begin{prop}
Let $V$ be a matrix-valued operator satisfying Hypothesis $\ref{hyp-1}$ and let $p\in [1,+\infty)$. Then, there exists a positive constant $ C = C(p,d,m) $ such that 
\begin{align}\label{equation2.1}
&\int_{Q}(\norm{\nabla \uu(x)}^p + \langle V(x)\uu(x),\uu(x)\rangle\norm{\uu(x)}^{p-2})\,dx
\notag\\
\ge &\frac{1}{2^{p-1}}\operatorname{av}_Q(\min\{CR^{-p},\lambda_V\})\int_{Q}\norm{\uu(x)}^pdx,
\end{align}
for every cube $Q\subset \R^d$ with side-length $R$ and every $ \uu\in W^{1,p}_{\rm loc}(\R^d,\R^m)$ such that
$\langle V\uu,\uu\rangle\|\uu\|^{p-2}$ belongs to $L^1_{\rm loc}(\R^d,\R^m)$.
\end{prop}

\begin{proof}
We fix a cube $Q\subset\R^d$ with side-length $R$, $p\in [1,+\infty)$ and, to begin with, we prove that, for every $p\in [1,+\infty)$ and every $\uu\in W^{1,p}_{\rm loc}(\R^d,\R^m)$, there exists a positive constant $C(p,d,m)$ such that
\begin{align}
\int_{Q}\norm{\nabla\uu(x)}^pdx \ge  \frac{C(p,d,m)}{R^{d+p}}\int_{Q\times Q}\norm{\uu(x)-\uu(y)}^pdxdy.
\label{leao}
\end{align}
Clearly, it suffices to prove \eqref{leao} for functions $\uu\in C^1(\R^d,\R^m)$, since a straightforwards density argument allows us to extend its validity to any $\uu\in W^{1,p}_{\rm loc}(\R^d,\R^m)$.\\
So, let assume that $\uu\in C^1(\R^d,\R^m)$ and set $Q=\displaystyle\prod_{j=1}^d[a_j,a_j+R]$ for some $a_1,\ldots,a_d\in\R$.
For $x=(x_1,x_2,\ldots ,x_d),\,y=(y_1,y_2,\ldots ,y_d)\in Q$, define
\begin{eqnarray*}
x^{(0)}\coloneqq x,\quad x^{(k)}\coloneqq(y_1,\ldots ,y_k,x_{k+1},\ldots ,x_d) \hbox{\ for }1\le k\le d-1,\quad x^{(d)}\coloneqq y.
\end{eqnarray*}
Since $u_j\in C^1(\R^d)$ for every $j=1,\ldots,m$, it follows that
\begin{align*}
u_j(x)-u_j(x^{(1)})&=\int_{y_{1}}^{x_{1}}\partial_{1}u_j(t,x_2,\ldots ,x_d)\,dt,\\
u_j(x^{(k)})-u_j(x^{(k+1)})&=\int_{y_{k+1}}^{x_{k+1}}\partial_{k+1}u_j(y_1,\ldots ,y_k,t,x_{k+2},\ldots ,x_d)\,dt,\qquad\;\, 1\le k\le d-2,\\
u_j(x^{(d-1)})-u_j(y)&=\int_{y_d}^{x_d}\partial_{d}u_j(y_1,y_2,\ldots ,y_{d-1},t)\,dt.
\end{align*}
H\"older's inequality gives
\begin{align*}
&|u_j(x^{(k)})-u_j(x^{(k+1)})|^p\\
\le &|x_{k+1}-y_{k+1}|^{p-1}\int_{y_{k+1}}^{x_{k+1}} \left|\partial_{k+1}u_j(y_1,\ldots ,y_k,t,x_{k+2},\ldots ,x_d)\right|^pdt
\end{align*}
for $k=0,\ldots,d-2$ and
\begin{align*}
|u_j(x^{(d-1)})-u_j(x^{(d)})|^p\le |x_d-y_d|^{p-1}\int_{y_d}^{x_d} \left|\partial_du_j(y_1,\ldots ,y_k,t)\right|^pdt.
\end{align*}

Hence,
\begin{align*}
|u_j(x)-u_j(y)|^p\le &\left(\sum_{k=0}^{d-1}|u_j(x^{(k)})-u_j(x^{(k+1)})|\right)^p\\
\le & C(p,q)\sum_{k=0}^{d-1}|u_j(x^{(k)})-u_j(x^{(k+1)})|^p\\
\le & C(p,d)R^{p-1}\sum_{k=0}^{d-1}\bigg |\int_{y_{k+1}}^{x_{k+1}} \left|\partial_{k+1}u_j(y_1,\ldots ,y_k,t,x_{k+2},\ldots ,x_d)\right|^pdt\bigg |,
\end{align*}
where, with a slight abuse of notation, 
$(y_1,\ldots ,y_k,t,x_{k+2},\ldots ,x_d)=
(y_1,\ldots ,y_{d-1},t)$ when $k=d-1$.
Thus,
\begin{align*}
&\int_Q|u_j(x)-u_j(y)|^p\,dx\\
\le & C(p,d)R^{p-1}\int_Q\sum_{k=0}^{d-1}\bigg |\int_{y_{k+1}}^{x_{k+1}} \left|\partial_{k+1}u_j(y_1,\ldots ,y_k,t,x_{k+2},\ldots ,x_d)\right|^pdt\bigg |dx\\
\le &C(p,d)R^{p-1}\int_Qdx\sum_{k=0}^{d-1}\int_{a_{k+1}}^{a_{k+1}+R} \left\|\nabla u_j(y_1,\ldots ,y_k,t,x_{k+2},\ldots ,x_d)\right\|^p\,dt\\
\le & C(p,d)R^{p-1}\sum_{k=0}^{d-1}R^{k+1}\int_{Q_{k+2}}dx_{k+2}\cdots dx_d\\
&\qquad\qquad\qquad\qquad\qquad\qquad\;\times\int_{a_{k+1}}^{a_{k+1}+R} \left\|\nabla u_j(y_1,\ldots ,y_k,t,x_{k+2},\ldots ,x_d)\right\|^pdt\\
=&C(p,d)R^{p}\sum_{k=0}^{d-1}R^{k}\int_{Q_{k+1}}\left\|\nabla u_j(y_1,\ldots ,y_k,x_{k+1},x_{k+2},\ldots ,x_d)\right\|^pdx_{k+1}\cdots dx_d,
\end{align*}
where $Q_j=\prod_{i=j}[a_i,a_i+R]$ for every $j\le d$.
		
Now, integrating over $Q$ with respect to the variable $y$, we can write
\begin{align*}
&\int_{Q\times Q}|u_j(x)-u_j(y)|^p\,dxdy\\
\le &C(p,d)R^{p}\sum_{k=0}^{d-1}R^{k}\int_Qdy\int_{Q_{k+1}}\left\|\nabla u_j(y_1,\ldots ,y_k,x_{k+1},x_{k+2},\ldots ,x_d)\right\|^pdx_{k+1}\cdots dx_d\\
=& C(p,d)R^{p+d}\sum_{k=0}^{d-1}\int_Q\left\|\nabla u_j(y_1,\ldots ,y_k,x_{k+1},x_{k+2},\ldots ,x_d)\right\|^p dy_1\cdots dy_kdx_{k+1}\cdots dx_d\\
=&C(p,d)dR^{p+d}\int_Q\|\nabla u_j(x)\|^p\,dx.
\end{align*}
Therefore, the following estimate holds:
\begin{equation}\label{FF-argument-1}
\int_{Q}\norm{\D u_j(x)}^p\,dx \ge\frac{\tilde{C}(p,d)}{R^{d+p}}\int_{Q\times Q}\abs{u_j(x)-u_j(y)}^pdxdy.
\end{equation}
Since
\begin{align*}
\norm{\nabla\uu}^p\ge C(p,m)\sum_{j=1}^m\norm{\nabla u_j}^p,
\end{align*}
by applying \eqref{FF-argument-1} we obtain 
\begin{align*}
\int_{Q}\norm{\nabla\uu(x)}^pdx &\ge \frac{C(p,d,m)}{R^{d+p}}\sum_{j=1}^m\int_{Q\times Q}\abs{u_j(x)-u_j(y)}^pdxdy\\
&\ge  \frac{C(p,d,m)}{R^{d+p}}\int_{Q\times Q}\norm{\uu(x)-\uu(y)}^pdxdy
\end{align*}
and \eqref{leao} follows easily.

We now observe that if $\uu$ is such that
$\langle V\uu,\uu\rangle\|\uu\|^{p-2}\in L^1_{\rm loc}(\R^d)$, then
\begin{align*}
\int_{Q}\langle V(x)\uu(x),\uu(x) \rangle\|\uu(x)\|^{p-2}dx &\ge  \int_{Q}\lambda_V(x)\norm{\uu(x)}^pdx
=\frac1{R^d}\int_{Q\times Q}\lambda_V(x)\norm{\uu(x)}^pdxdy.
\end{align*}
Combining this inequality and \eqref{leao}, we get
\begin{align*}
&\int_{Q}(\norm{\nabla \uu(x)}^p + \langle V(x)\uu(x),\uu(x)\rangle\norm{\uu(x)}^{p-2})dx\\
\ge &\frac{C(p,d,m)}{R^{d+p}}\int_{Q\times Q}\norm{\uu(x)-\uu(y)}^pdxdy + \frac1{R^d}\int_{Q\times Q}\lambda_V(x)\norm{\uu(x)}^pdxdy \\
\ge &\frac1{R^d}\int_{Q\times Q}\min\{ C(p,d,m)R^{-p},\lambda_V(x)\}(\norm{\uu(x)-\uu(y)}^p+\norm{\uu(x)}^p)dxdy\\ 
\ge &\left(\frac1{R^d}\int_{Q}\min\left\lbrace C(p,d,m)R^{-p},\lambda_V(x)\right\rbrace dx \right)\left(\frac1{2^{p-1}}\int_{Q}\norm{\uu(y)}^p\,dy\right)\\[8pt] 
=&\frac{1}{2^{p-1}}\operatorname{av}_{Q}\left(\min\left\lbrace C(p,d,m)R^{-p},\lambda_V(\cdot)\right\rbrace\right) \left(\int_{Q}\norm{\uu(y)}^pdy\right),
\end{align*}
where we have used the inequality $ \frac1{2^{p-1}}|a|^p\le |a-b|^p + |b|^p$, which holds true for every $ a,b\in\R $ and $ 1\le p<\infty$.
\end{proof}

\section{Vector-valued Schr\"odinger operator}
\label{sect-defH}
Let us introduce the set 
\begin{align*}
\mathcal{V}=\{\f = (f_1,\dots,f_m)\in L^2(\R^d,\R^m): \nabla\f\in L^2(\R^d,\R^{d\times m}),\ V^{\frac12}\f \in L^2(\R^d,\R^m)\},
\end{align*}
and on $ \mathcal{V}\times\mathcal{V} $ let us define the sesquilinear form
\begin{align*} 
a(\f,\bm{g}) = \int_{\R^d}\bigg (\sum_{i=1}^{m}\langle \nabla f_i(x),\nabla g_i(x) \rangle+ \langle V(x)\f(x), {\bm g}(x) \rangle\,\bigg )\,dx
\end{align*}
for $\f, \bm{g}\in\mathcal{V}$.
The domain $ \mathcal{V} $ equipped with the norm
\begin{align*}
\norm\f_{\mathcal{V}} = \left(\norm\f^2_2 + a(\f,\f)\right)^{\frac12}
\end{align*}
is a Hilbert space and $C_c^\infty(\R^d,\R^m)$ is dense in $ \mathcal{V}$ (for more details we refer to \cite{dallara}).  Hence, since $a$ is accretive and continuous, there exists a unique nonnegative and self-adjoint operator $\H:D(\H)\to L^2(\R^d,\R^m)$,  defined by
\begin{align*}
D(\H) & =\{\uu\in \mathcal{V}\ {\rm such\  that}\ \exists \, \bm{v}\in L^2(\R^d,\R^m): a(\uu,{\bf \phi})=\langle \bm{v},{\bf \phi}\rangle_{L^2(\R^d,\R^m)}\;\forall {\bf \phi}\in \mathcal{V}\}, \\
\H\uu & =\bm{v}, \qquad \forall \uu\in D(\H).
\end{align*}

We notice that for any $\eps>0$ the operator $\H+\eps$ is invertible, but $\H$ itself in general is not invertible since the form $a$ might be not coercive. For this reason we introduce a version of $\H$ which is now invertible, in the sense of distributions, but it is defined in a larger space.

\subsection{Homogeneous version of \texorpdfstring{$\H$}{H}}
Let $ \dot{\mathcal{V}} $ be the closure of $ C_c^{\infty}(\R^d,\R^m)$ with respect to the norm
\begin{align*}
\norm\f_{\dot{\mathcal{V}}}= a(\f,\f)^{\frac12}=\left(\sum_{i=1}^m\int_{\R^d}(\norm{\nabla f_i(x)}^2 + \langle V(x)\f(x),\f(x)\rangle)dx\right)^{\frac12},\quad f\in \dot{\mathcal{V}}.
\end{align*}

Clearly, $\|\cdot\|_{\dot{\mathcal{V}}}$ is a seminorm.  To prove that, actually, it is a norm, it suffices to observe that for any $\f\in\dot{\mathcal{V}}$, estimate \eqref{equation2.1} with $p=2$ implies that $\f$ belongs to $L^2(Q,\R^m)$ and $\norm\f_{L^2(Q,\R^m)}\le c_Q\norm{\f}_{\dot{\mathcal{V}}}$ for every cube $Q\subset\R^d$, and this means that $ \dot{\mathcal{V}}\subset L^2_{\text{loc}}(\R^d,\R^m)$. This gives us that $\norm{\cdot}_{\dot{\mathcal{V}}}$ is a norm and $ a $ is the inner product associated to this norm. We can conclude that $(\dot{\mathcal{V}},
\norm{\cdot}_{\dot{\mathcal{V}}}) $ is a Hilbert space and $a$ is coercive in $\dot{\mathcal V}$.
   
If we choose not to identify $ \dot{\mathcal{V}} $ with its dual space, $  \dot{\mathcal{V}}'$, by Lax-Milgram's theorem, there exists a unique bounded and invertible operator $\dot{\H}\colon\dot{\mathcal{V}}\to\dot{\mathcal{V}}' $ such that $\langle \dot{\H}\,\uu,\bm{v} \rangle_{\dot{\mathcal{V}'},\dot{\mathcal{V}}} = a(\uu,\bm{v}),$ for every $\uu, \bm{v} \in \dot{\mathcal{V}}$.	
This means that for every $ \f\in\dot{\mathcal{V}}'$ there exists a unique $ \uu\in\dot{\mathcal{V}} $ such that $ a(\uu,\bm{v}) 
=\langle \f,\bm{v} \rangle_{\dot{\mathcal V}',\dot{\mathcal V}}$, 
for all $ \bm{v}\in C_c^\infty(\R^d,\R^m)$.  
Then, $ -\Delta \uu  + V\uu = \f,$ in the sense of distributions with $\uu = \dot{\H}^{-1}\f$.

\begin{rem}
\label{rem-EURO24}
Since $C^{\infty}_c(\R^d,\R^m)$ is dense both in $\mathcal{V}$ and in $\dot{\mathcal V}$, it follows that $\mathcal V\subset\dot{\mathcal V}$. Further, if the minimum eigenvalue of the matrix $V$ satisfies the condition $\lambda_V(x)\ge C$ for almost every $x\in\R^d$ and some positive constant $C$, then the spaces $\mathcal{V}$ and $\dot{\mathcal{V}}$ actually coincide.
Indeed, if $\f\in \dot{\mathcal V}$ then there exists a sequence $(\f_n)_{n\in\N}\subset C^{\infty}_c(\R^d,\R^m)$ such that $(\partial_j\f_n)_{n\in\N}$ and $(V^{1/2}\f_n)_{n\in\N}$ converge, respectively, to $\partial_j\f$ and $V^{1/2}\f$ in $L^2(\R^d,\R^m)$, as $n$ tends to $\infty$, for every $j=1,\ldots,d$. Since $\|V^{1/2}(\f_n-\f)\|^2\ge\lambda_V\|\f_n-\f\|^2\ge C\|\f_n-\f\|^2$ for every $n\in\N$, the sequence $(\f_n)_{n\in\N}$ converges to $\f$ in $L^2(\R^d,\R^m)$ as $n$ tends to $+\infty$, so that $\f\in\mathcal{V}$.

In view of this property, we will simply write $\mathcal{V}$ instead of $\dot{\mathcal V}$, when the function $x\mapsto\lambda_V(x)$ is bounded from below by
a positive constant.
\end{rem}

The following lemma provides us with an useful approximation result.

\begin{lemma}\label{lem:3.1}
Fix $\f\in L^2(\R^d,\R^m)\cap\dot{\mathcal{V}}'$ and, for every $\eps>0$, define $ \uu_{\eps} = (\H+\eps)^{-1}\f\in D(\H).$ Then, the family $(\uu_{\eps})_{\eps>0} $  converges to $\dot{\H}^{-1}\f$ in $\dot{\mathcal V}$ as $ \eps$ tends to $0$.
\end{lemma}
	
\begin{proof}
Let $ \f$ and $\uu_{\eps}$ be as in the statement. By definition we get
\begin{equation}\label{eq:def:ueps:form}
\int_{\R^d}\left(\sum_{i=1}^m\langle \nabla (u_{\eps})_i,\nabla v_i\rangle + \langle(V+\eps)\uu_{\eps},\bm{v}\rangle\right)dx = \int_{\R^d}\langle \f,\bm{v}\rangle\, dx,\qquad\;\,\bm{v}\in\mathcal{V}.
\end{equation} 
Now we can choose $ \bm{v} = \uu_{\eps}\in D(\H)\subseteq\mathcal{V}$ and obtain that
\begin{align*}
\int_{\R^d}\left(\sum_{i=1}^m \norm{\nabla (u_{\eps})_i}^2 + \langle(V+\eps)\uu_{\eps},\uu_{\eps}\rangle\right)dx = \int_{\R^d}\langle
\f,\uu_{\eps}\rangle\,dx
\end{align*}
for every $\varepsilon>0$. Then,
\begin{align}
\norm{\uu_{\eps}}^2_{\dot{\mathcal{V}}} &\le \int_{\R^d}\bigg (\sum_{i=1}^m \norm{\nabla (u_{\eps})_i}^2 + \langle(V+\eps)\uu_{\eps},\uu_{\eps}\rangle\bigg )dx= \int_{\R^d}\langle \f,\uu_{\eps}\rangle\,dx \notag \\
&\le \left|\int_{\R^d}\langle \f,\uu_{\eps}\rangle\,dx\right|\le\norm{\f}_{\dot{\mathcal{V}}'}\norm{\uu_{\eps}}_{\dot{\mathcal{V}}}
\label{lim_u_eps}
\end{align}
for every $\varepsilon>0$, and we can conclude that $\norm{\uu_{\eps}}_{\dot{\mathcal{V}}}\le\norm{\f}_{\dot{\mathcal{V}}'}$, i.e., $ (\uu_{\eps})_{\eps>0} $ is bounded in $ \dot{\mathcal{V}}$. Hence, there exists a sequence $(\varepsilon_n)_{n\in\N}$, converging to zero, such that 
$ (\uu_{\eps_n})_{n\in\N}$ weakly converges to some function $\uu\in\dot{\mathcal{V}}$ as $\varepsilon$ tends to $0$.
Writing \eqref{eq:def:ueps:form} with $\uu_{\varepsilon_n}$ instead of $\uu_{\varepsilon}$ and taking the limit as $n$ tends to $\infty$, we obtain
\begin{align*}
\int_{\R^d}\left(\sum_{i=1}^m\langle \nabla u_i,\nabla v_i\rangle + \langle V\uu,\bm{v}\rangle\right)dx = \int_{\R^d}\langle \f,\bm{v}\rangle\,dx,\qquad\;\, \bm{v}\in {\mathcal{V}}.
\end{align*}
Since $C_c^\infty(\R^d,\R^m)$ is dense both in $\mathcal V$ and in $\dot{\mathcal V}$, by the definition of the operator $\dot{\H}$ we conclude that $\uu=\dot{\H}^{-1}\f$.

To prove that $(\uu_{\varepsilon_n})_{n\in \N}$ strongly converges to $\uu$ it is enough to prove the convergence of the norms. For this purpose, we observe that
\begin{align*}
\norm{\uu}^2_{\dot{\mathcal{V}}} &= \int_{\R^d}\bigg (\sum_{i=1}^m \norm{\nabla u_i}^2 + \langle V\uu,\uu\rangle\bigg )dx \\
&\le \liminf_{n\to+\infty} \int_{\R^d}\bigg (\sum_{i=1}^m \norm{\nabla (u_{\eps_n})_i}^2 + \langle V\uu_{\eps_n},\uu_{\eps_n}\rangle\bigg )dx \\
&\le \limsup_{n\to +\infty} \int_{\R^d}\bigg (\sum_{i=1}^m \norm{\nabla (u_{\eps_n})_i}^2 + \langle V\uu_{\eps_n},\uu_{\eps_n}\rangle\bigg )dx \\
&\le \limsup_{n\to+\infty} \int_{\R^d}\bigg (\sum_{i=1}^m \norm{\nabla (u_{\eps_n})_i}^2 + \langle (V+\eps_n)\uu_{\eps_n},\uu_{\eps_n}\rangle\bigg )dx \\
&=\limsup_{n\to +\infty}\int_{\R^d}\langle \f,\uu_{\eps_n}\rangle\,dx\\
&= \int_{\R^d}\langle \f,\uu\rangle\,dx \\
&=  \int_{\R^d}\bigg (\sum_{i=1}^m\langle \nabla u_i,\nabla u_i\rangle + \langle V\uu,\uu\rangle\bigg )dx =  \norm{\uu}^2_{\dot{\mathcal{V}}},
\end{align*}	
so that $(\norm{\uu_{\eps_n}}_{\dot{\mathcal{V}}})_{n\in\N}$ converges to $\norm{\uu}_{\dot{\mathcal{V}}}$ as $\varepsilon$ tends to $0$.  

To conclude the proof, we observe that the previous results show that every convergent subsequence of $(\uu_{\varepsilon})_{\eps>0}$ converges to $\uu$. Hence, all the family $(\uu_{\varepsilon})_{\eps>0}$ converges to $\uu$ in $\dot{\mathcal V}$ as $\varepsilon$ tends to $0$. Indeed, assume by contradiction that there exists a sequence $(\uu_{\varepsilon_n})_{n\in\N}$, with $(\varepsilon_n)_{n\in\N}$ decreasing to $0$, which does not converge to $\uu$. It follows that there exist $\delta>0$ and a subsequence $(\uu_{\varepsilon_{n_k}})_{k\in\N}\subseteq (\uu_{\varepsilon_n})_{n\in\N}$ such that 
\begin{align}
\label{conv_uu_eps_n_k}
\|\uu_{\varepsilon_{n_k}}-\uu\|_{\dot{\mathcal V}}\geq \delta, \qquad k\in\N.
\end{align}
But formula \eqref{lim_u_eps}, with $\uu_\varepsilon$ replaced by $\uu_{\varepsilon_{n_k}}$, implies that $(\uu_{\varepsilon_{n_k}})_{k\in\N}$ is bounded in $\dot{\mathcal V}$. Hence, it admits a subsequence which weakly converges to some $\bm{v}$ in $\dot{\mathcal V}$ and, arguing as above, we infer that such subsequence strongly converges to $\dot{\mathcal H}^{-1}\f$. This contradicts \eqref{conv_uu_eps_n_k}, and so we obtain that the whole family $(\uu_{\varepsilon})_{\eps>0}$ converges to $\uu$ in $\dot{\mathcal V}$ as $\varepsilon$ tends to $0$.
\end{proof}

 \begin{rem}\label{rem:3.2}
Since $ \dot{\mathcal{V}}\subset L^2_{\text{loc}}(\R^d,\R^m) $, taking 
formula \eqref{equation2.1}, with $p=2$, into account, it follows easily that the embedding $ L^2_c(\R^d,\R^m)\subset\dot{\mathcal{V}}'$ is continuous.
Moreover,  $(\uu_{\eps})_{\eps>0} $ converges to $\dot{\mathcal H}^{-1}\f$ in $ L^2(\Omega,\R^m)$ for every bounded measurable set $\Omega\subset\R^d$. It thus follows that $(\uu_{\eps})_{\eps>0}$ has a subsequence converging to $\dot{\mathcal H}^{-1}\f$ almost everywhere in $\R^d$.
\end{rem}
	
\section{\texorpdfstring{$L^1$}{L1}-maximal inequalities}
\label{sec:L1}
We start this section with a vectorial version of Kato's inequality. This result has been obtained in \cite[Proposition 2.3]{kunze-maichine-rhandi:2018} assuming that $\uu\in H^1_{\rm loc}(\R^d,\R^m)$ and $\Delta\uu\in L^1_{\rm loc}(\R^d,\R^m)$. Adapting the technique in
the proof of \cite[Lemma A]{kato72}, which deals with the scalar case, the condition $\uu\in H^1_{\rm loc}(\R^d,\R^m)$ can be weakened, assuming that $\uu\in L^1_{\rm loc}(\R^d,\R^m)$. We provide the details for reader's convenience.

\begin{lemma}
\label{lem:kato_vett}
If $\bm{u}$ and $\Delta\bm{u}$ belong to $L^1_{\rm loc}(\R^d,\R^m)$, then 
$\Delta \|\bm{u}\|\geq\chi_{\{\bm{u}\neq \bm{0}\}}\|\bm{u}\|^{-1}\langle \bm{u},\Delta\bm{u}\rangle$ in the sense of distributions.     
\end{lemma}
\begin{proof}
Let $\uu\in C^2(\R^d,\R^m)$ and set $u_{\varepsilon}\coloneqq(\|\uu\|^2+\varepsilon^2)^{1/2}$ for every $\varepsilon>0$. Explicit computations give
\begin{equation}
\label{u_eps_kato}
2u_\varepsilon\partial _k(u_\varepsilon)
=  \partial_k(u_\varepsilon)^2
= 2\sum_{i=1}^mu_i\partial_ku_i, \qquad k=1,\ldots,d.
\end{equation}
Hence,
\begin{align*}
(u_\varepsilon)^2\sum_{k=1}^d(\partial_ku_\varepsilon)^2
= \sum_{k=1}^d\langle \uu,\partial_k \uu\rangle^2,
\end{align*}
which implies that 
\begin{align}
\label{u_eps_kato_2}
\sum_{k=1}^d(\partial_ku_\varepsilon)^2\leq \frac{\|\uu\|^2} {(u_\varepsilon)^2}\sum_{k=1}^d\|\partial_k\uu\|^2\leq \sum_{k=1}^d\|\partial_k\uu\|^2.     
\end{align}
Differentiating \eqref{u_eps_kato} with respect to $k\in\{1,\ldots,d\}$ and summing up $k$ from $1$ to $d$ we get
\begin{align*}
u_\varepsilon \Delta u_\varepsilon+\sum_{k=1}^d(\partial_ku_\varepsilon)^2=\langle \uu,\Delta\uu\rangle+\sum_{k=1}^d\|\partial_k\uu\|^2,  
\end{align*}
and from \eqref{u_eps_kato_2} we infer that $\Delta u_\varepsilon\geq (u_\varepsilon)^{-1}\langle \uu,\Delta\uu\rangle$. Now we extend this inequality to every $\uu\in L^1_{\rm loc}(\R^d,\R^m)$ such that $\Delta\uu\in L^1_{\rm loc}(\R^d,\R^m)$. We consider a family of mollifiers $(\rho_\theta)_{\theta>0}$ and we set $\uu^\theta=((u^\theta)_1,\ldots,(u^\theta)_m)$ with  $(u^\theta)_i=u_i*\rho_\theta$ for every $i=1,\ldots,m$ and $\theta>0$, and 
\begin{align*}
u_\varepsilon\coloneqq(\|\uu\|^2+\varepsilon^2)^{1/2}, \qquad u^\theta_\varepsilon\coloneqq(\|\uu^\theta\|^2+\varepsilon^2)^{1/2}    
\end{align*}
for every $\varepsilon,\theta>0$. We get
\begin{align*}
|u^\theta_\varepsilon-u_\varepsilon|\leq |\|\uu^\theta\|-\|\uu\||
\leq \|\uu^\theta-\uu\|, \qquad \varepsilon,\theta>0.
\end{align*}
Since $(\uu^\theta)_{\theta>0}$ converges to $\uu$ in $L^1_{\rm loc}(\R^d,\R^m)$, it follows that $(u^\theta_\varepsilon)_{\theta>0}$ converges to $u_\varepsilon$ in $L^1_{\rm loc}(\R^d)$. Hence, $(\Delta (u^\theta_\varepsilon))_{\theta>0}$ converges to $\Delta u_\varepsilon$ in $\mathcal D'$. Further,
\begin{align*}
(u^\theta_\varepsilon)^{-1}\langle \uu^\theta,\Delta\uu^\theta\rangle -(u_\varepsilon)^{-1}\langle \uu,\Delta\uu\rangle
\!= &  (u^\theta_\varepsilon)^{-1}\langle \uu^\theta,\Delta\uu^\theta-\Delta\uu\rangle  + \langle (u^\theta_\varepsilon)^{-1}\uu^\theta-(u_\varepsilon)^{-1}\uu,\Delta\uu\rangle
\end{align*}
for every $\varepsilon,\theta>0$. Since $\|(u^\theta_\varepsilon)^{-1}\uu^\theta\|\leq 1$, $\|(u_\varepsilon)^{-1}\uu\|\leq 1$ and $(u^\theta_\varepsilon)^{-1}\uu^\theta-(u_\varepsilon)^{-1}\uu$ pointwise vanishes as $\theta$ goes to $0$, by dominated convergence we get 
\begin{align*}
(u^\theta_\varepsilon)^{-1}\langle \uu^\theta,\Delta\uu^\theta-\Delta\uu\rangle  + \langle (u^\theta_\varepsilon)^{-1}\uu^\theta-(u_\varepsilon)^{-1}\uu,\Delta\uu\rangle\to 0    
\end{align*}
in $L^1_{\rm loc}(\R^d)$ as $\theta$ goes to $0$. Hence, we get $\Delta u_\varepsilon\geq (u_\varepsilon)^{-1}\langle \uu,\Delta\uu\rangle$ for every $\uu\in L^1_{\rm loc}(\R^d,\R^m)$ with $\Delta\uu\in L^1_{\rm loc}(\R^d,\R^m)$.

To conclude, we take the limit as $\varepsilon$ goes to $0$ in $\Delta u_\varepsilon\geq (u_\varepsilon)^{-1}\langle \uu,\Delta\uu\rangle$. We notice that $(u_\varepsilon)_{\varepsilon>0}$ monotonically converges to $\|\uu\|$, that $\|(u_\varepsilon)^{-1}\uu\|\leq 1$ and that $((u_\varepsilon)^{-1}\uu)_{\varepsilon>0}$ pointwise tends to $\chi_{\{\uu\neq\bm{0}\}}\|\uu\|^{-1}\uu$ as $\varepsilon$ vanishes. Letting $\varepsilon$ go to $0$, it follows that $\Delta\|\uu\|\geq \chi_{\{\uu\neq\bm{0}\}}\|\uu\|^{-1}\langle \uu,\Delta\uu\rangle$ in the sense of distributions.
\end{proof}

In the proof of Theorem \ref{prop:stime_L1} we will need the following auxiliary results.

\begin{lemma}\label{lem:rel_VepsM_Veps}
For every $\varepsilon>0$ and every $M\in\N$, let $V_{\varepsilon,M}$ be the matrix whose entries are $(V_{\varepsilon,M})_{ii}=v_{ii}+\varepsilon$ and $(V_{\varepsilon,M})_{ij}=v_{ij}\vee(-M)$, for every $i,j=1,\ldots,m$ with $i\neq j$. Moreover, set $V_{\varepsilon}=V+\varepsilon$.
Then,
\begin{align}
\langle V_{\varepsilon,M}\xi,\xi\rangle
\geq \langle V_{\varepsilon,M}|\xi|,|\xi|\rangle\geq \langle V_{\varepsilon}|\xi|,|\xi|\rangle
\label{stima_V}
\end{align} 
for all $ \xi\in\R^m$. In particular, if $\lambda_{V_{\varepsilon,M}}$ denotes the minimum eigenvalue of $V_{\varepsilon,M}$ and $\lambda_{V_{\varepsilon}}$ denotes the minimum eigenvalue of $V_\varepsilon$, then it follows that $\lambda_{V_{\varepsilon,M}}\geq \lambda_{V_{\varepsilon}}\ge\varepsilon$ for every $\varepsilon>0$ and every $M\in\N$.
\end{lemma}

\begin{proof}
We fix $\xi\in\R^m$ and observe that the first inequality follows from the fact that the off-diagonal entries of the matrix $V_{\varepsilon,M}$ are nonpositive. Indeed,
\begin{align*}
\langle V_{\varepsilon,M}\xi,\xi\rangle 
&=\sum_{i=1}^m(V_{\varepsilon,M})_{ii}\xi^2_i + \sum_{i,j=1,i\ne j}^m(V_{\varepsilon,M})_{ij}\xi_i\xi_j\\
&\ge\sum_{i=1}^m(V_{\varepsilon,M})_{ii}|\xi_i|^2 + \sum_{i,j=1,i\ne j}^m(V_{\varepsilon,M})_{ij}|\xi_i||\xi_j|\\
&=\langle V_{\varepsilon,M}|\xi|,|\xi|\rangle.
\end{align*}
To prove the second inequality, we notice that the matrices $ V_{\eps,M} $ and $ V_{\eps} $ have the same elements on the main diagonal, whereas $ (V_{\eps})_{ij}=v_{ij} $ for $ i\ne j$. Therefore, 
\begin{align*}
\langle V_{\varepsilon,M}|\xi|,|\xi|\rangle  
&=\sum_{i=1}^m(V_{\varepsilon})_{ii}|\xi_i|^2 + \sum_{i,j=1,i\ne j}^m(v_{ij}\vee(-M))|\xi_i||\xi_j|\\
&\ge\sum_{i=1}^m(V_{\varepsilon})_{ii}|\xi_i|^2 + \sum_{i,j=1,i\ne j}^mv_{ij}|\xi_i||\xi_j|\\
&=\langle V_{\varepsilon}|\xi|,|\xi|\rangle.
\end{align*}
The proof of \eqref{stima_V} is complete. 

The last statement now follows immediately from \eqref{stima_V}.
\end{proof}

\begin{lemma}
\label{lemma-V-eps-M-N}
For every $\varepsilon>0$, every $M,N\in\N$ and every $x\in\R^d$, let $V_{\varepsilon,M,N}(x)$ be the matrix whose entries are $(V_{\varepsilon,M,N}(x))_{ii}=(v_{ii}(x)+\varepsilon)\wedge N$ and $(V_{\varepsilon,M,N}(x))_{ij}=v_{ij}(x)\vee(-M)$, for every $i,j=1,\ldots,m$ with $i\neq j$. 
Then, for every fixed $M\in\N$ and $\varepsilon>0$, there exists an integer $N(\varepsilon,M,m)$ such that $V_{\varepsilon,M,N}(x)\ge \frac{\varepsilon}{2}$ in the sense of forms for every $x\in\R^d$.
\end{lemma}

\begin{proof}
We fix $x\in \R^d$ and set
\begin{align*}
I_1(x)= &  \{i\in\{1,\ldots,m\}:v_{ii}(x)+\varepsilon\leq N\}, \\ 
I_2(x)=&\{1,\ldots,m\}\setminus I_1(x).    
\end{align*}

Recalling that $v_{ij}\le 0$ for all $i\neq j$, we can estimate
\begin{align}
\sum_{i,j\in I_1}(V_{\eps,M,N}(x))_{ij}\xi_i\xi_j
= & \sum_{i\in I_1}(v_{ii}(x)+\varepsilon)\xi_i^2
+ \sum_{i,j\in I_1,j\neq i}[v_{ij}(x)\vee (-M)]\xi_i\xi_j \notag \\
\geq & \sum_{i\in I_1}v_{ii}(x)\xi_i^2+\sum_{i,j\in I_1,j\neq i}v_{ij}(x)|\xi_i||\xi_j| +\varepsilon\sum_{i\in I_1}\xi_i^2 \notag \\
\geq  & \varepsilon\sum_{i\in I_1}\xi_i^2,
\label{stima_appr_somma_1}
\end{align}
since 
\begin{align*}
\sum_{i\in I_1}v_{ii}(x)\xi_i^2+\sum_{i,j\in I_1,j\neq i}v_{ij}(x)|\xi_i||\xi_j|
= \langle V(x)\eta,\eta\rangle\geq0,
\end{align*}
where the entries of the vector $\eta\in \R^m$ are $\eta_i=|\xi_i|$ if $i\in I_1$ and $\eta_i=0$ otherwise.
		
Next, we note that 
\begin{align}
\sum_{i,j\in I_2}(V_{\eps,M,N}(x))_{ij}\xi_i\xi_j
= & N\sum_{i\in I_2}\xi_i^2
+ \sum_{i,j\in I_2,i\neq j}[v_{ij}(x)\vee(-M)]\xi_i\xi_j \notag \\
\geq & N\sum_{i\in I_2}\xi_i^2
- M\sum_{i,j\in I_2,i\neq j}|\xi_i||\xi_j| \notag \\
\ge & N\sum_{i\in I_2}\xi_i^2
- M\bigg (\sum_{i\in I_2}|\xi_i|\bigg )^2\notag\\
\geq & (N-mM)\sum_{i\in I_2}\xi_i^2,
\label{stima_appr_somma_2}
\end{align}
where we have applied the Cauchy-Schwartz inequality in the last step. 

Finally, applying the Young inequality we can infer that
\begin{align}
\sum_{i\in I_1,j\in I_2}(V_{\eps,M,N}(x))_{ij}\xi_i\xi_j
= & \sum_{i\in I_1,j\in I_2}[v_{ij}(x)\vee(-M)]\xi_i\xi_j
\geq -M\sum_{i\in I_1,j\in I_2}|\xi_i||\xi_j| \notag \\
\geq & -M \sum_{i\in I_1,j\in I_2}\left(\frac{\varepsilon}{4mM}\xi_i^2+\frac{mM}{\varepsilon}\xi_j^2\right) \notag \\
\geq & -\frac{\varepsilon}{4}\sum_{i\in I_1}\xi_i^2-\frac{m^2M^2}{\varepsilon}\sum_{j\in I_2}\xi_j^2.
\label{stima_appr_somma_3}
\end{align}

From \eqref{stima_appr_somma_1}, \eqref{stima_appr_somma_2} and \eqref{stima_appr_somma_3} it follows that
\begin{align*}
\langle V_{\varepsilon,M,N}(x)\xi,\xi\rangle
\geq & \frac{\varepsilon}{2}\sum_{i\in I_1}\xi_i^2+\left(N-mM-\frac{2m^2M^2}{\varepsilon}\right)\sum_{j\in I_2}\xi_j^2\\
\ge & \min\left\{\frac{\varepsilon}{2},N-mM-\frac{2m^2M^2}{\varepsilon}\right\}|\xi|^2
\end{align*}
and the coefficient in the last of the previous chain of inequalities equals $\frac{\varepsilon}{2}$, provided that $N\geq \displaystyle N(\varepsilon,M,m)\coloneqq\left[\displaystyle mM+\frac{2m^2M^2}{\varepsilon}+\frac{\varepsilon}{2}\right]+2$. The assertion is proved.
\end{proof}

For every $\eps>0$ and $M\in\N$, we now introduce the operator ${\mathcal H}_{\varepsilon,M}$, associated to the form 
\begin{eqnarray*}
a_{\varepsilon,M}(\f,\bm{g}) = \int_{\R^d}\bigg (\sum_{i=1}^{m}\langle \nabla f_i(x),\nabla g_i(x) \rangle+ \langle V_{\varepsilon,M}(x)\f(x), {\bm g}(x) \rangle\bigg )dx
\end{eqnarray*}
for every $\f,\bm{g}\in {\mathcal V}_{\varepsilon,M}$, where $ {\mathcal V}_{\varepsilon,M}$ is the closure of $C^{\infty}_c(\R^d,\R^m)$ with respect to the norm induced by the form $a_{\varepsilon,M}$. $ {\mathcal V}_{\varepsilon,M}$ coincides with the set of all functions $\f\in H^1(\R^d,\R^m)$ such that $V^{1/2}_{\varepsilon,M}\f\in L^2(\R^d,\R^m)$ by Lemma \ref{lem:rel_VepsM_Veps} and Remark \ref{rem-EURO24}. Similarly, the space $ {\mathcal V}_{\varepsilon}$ is the closure of $C^{\infty}_c(\R^d,\R^m)$ with respect to the norm induced by the form $a_{\varepsilon}$, defined as $\sqrt{a_{\varepsilon}(\f,\f)} = (a(\f,\f) + \eps\|\f\|_2^2)^{1/2}$ for every $\f\in{\mathcal V}_{\varepsilon}$.

\begin{prop}
\label{prop:conv_u_eps_M}
For every $\f\in L^2_c(\R^d,\R^m)$ and $\eps>0$,
the function ${\mathcal H}_{\varepsilon,M}^{-1}\f$ converges in
$\mathcal{V}_{\varepsilon}$ to the function $({\mathcal H}+\varepsilon)^{-1}\f$ as $M$ tends to infinity.
In particular, if $\bm{f}$ is componentwise nonnegative, then ${\mathcal H}_{\varepsilon,M}^{-1}\bm{f}$ and 
$(\mathcal H+\varepsilon)^{-1}\bm{f}$ are componentwise nonnegative and ${\mathcal H}_{\varepsilon,M_1}^{-1}\bm{f}\leq {\mathcal H}_{\varepsilon,M_2}^{-1}\bm{f}$ componentwise for every $M_1\leq M_2$.
\end{prop}

\begin{proof}
Let us fix $\f\in L^2_c(\R^d,\R^m)$, componentwise nonnegative, $\eps>0$, and set ${\bm{u}}_{\varepsilon,M}={\mathcal H}_{\varepsilon,M}^{-1}\f$. Then, $\uu_{\varepsilon,M}\in {\mathcal V}_{\varepsilon,M}$ and 
\begin{align}
\int_{\R^d}\bigg (\sum_{i=1}^m\langle\nabla (u_{\varepsilon,M})_i,\nabla v_i\rangle+\langle V_{\varepsilon,M}\uu_{\varepsilon,M},\bm{v}\rangle\bigg )dx=\int_{\R^d}\langle\f,\bm{v}\rangle dx    
\label{1}
\end{align}
for every $\bm{v}\in  {\mathcal V}_{\varepsilon,M}$.

Observe that $\uu_{\varepsilon,M}$ is componentwise nonnegative. Indeed, ${\mathcal H}_{\varepsilon,M}$ can be split into the sum $${\mathcal H}_{\varepsilon,M}=\widetilde{\mathcal{H}}_{\varepsilon}
+\widetilde{V}_M,$$ where $(\widetilde{V}_M)_{ii}=0,\,(\widetilde{V}_M)_{ij}=v_{ij}\vee (-M)$ for $i\neq j$, and $\widetilde{\mathcal{H}}_{\varepsilon}$ is the diagonal operator associated to the form $\widetilde a_{\varepsilon}$, defined by
\[
\widetilde a_{\varepsilon}(\uu,\bm{v})=
\int_{\R^d}\bigg (\sum_{i=1}^m\langle\nabla u_i,\nabla v_i\rangle
+\sum_{i=1}^m(v_{ii}+\varepsilon)u_iv_i\bigg )dx
\]
for every $\uu,\bm{v}\in\widetilde{\mathcal V}_{\varepsilon}=\{\uu\in H^1(\R^d,\R^m): \sqrt{v_{jj}}u_j\in L^2(\R^d),\ j=1,\ldots,m\}$.
Since $0\le v_{ii}\in L^1_{\rm loc}(\R^d)$, and $\varepsilon >0$, it follows that $-\widetilde{\mathcal{H}}_{\varepsilon}$ generates a positive $C_0$-semigroup of contractions on $L^2(\R^d,\R^m)$ and the spectral bound $s(-\widetilde{\mathcal{H}}_{\varepsilon})$ is negative. Moreover, from the boundedness of the entries of the matrix-valued function $\widetilde{V}_M$ and the bounded perturbation theorem for semigroups, we deduce that $-{\mathcal H}_{\varepsilon,M}$ generates a $C_0$-semigroup $(e^{-t{\mathcal H}_{\varepsilon,M}})_{t\ge 0}$ on $L^2(\R^d,\R^m)$ given by the Trotter product formula 
\begin{eqnarray*} 
e^{-t{\mathcal H}_{\varepsilon,M}}\bm{f}=\lim_{n\to\infty}\left(e^{-\frac{t}{n}\widetilde{\mathcal{H}}_{\varepsilon}}
e^{-\frac{t}{n}\widetilde{V}_M}\right)^n\bm{f},\qquad\;\, \bm{f}\in L^2(\R^d,\R^m),\;\,t>0,
\end{eqnarray*}
see \cite[Corollary III.5.8]{EN}. 
On the other hand, since $(\widetilde{V}_M)_{ij}\le 0$ for all $i\neq j$, one deduces from \cite[Theorem 7.1]{BKR17} that the semigroup $(e^{-t\widetilde{V}_M})_{t\ge 0}$ is positive. Now, the positivity of $(e^{-t{\mathcal H}_{\varepsilon,M}})_{t\ge 0}$ follows from the positivity of $(e^{-t\widetilde{\mathcal{H}}_{\varepsilon}})_{t\ge 0}$ and the above Trotter product formula.
Furthermore, it follows from Lemma \ref{lem:rel_VepsM_Veps} that $a_{\varepsilon,M}(\f,\f)\ge \varepsilon \|\f\|_{L^2(\R^d,\R^m)}$ for every $\f\in  {\mathcal V}_{\varepsilon,M}$. Hence,
\begin{equation}\label{Marrak}
s(-\mathcal{H}_{\varepsilon ,M})<0.
\end{equation}
Using \eqref{Marrak} and \cite[Corollary 12.10]{BKR17}, one concludes that $\uu_{\varepsilon,M}\ge\bm{0}$.

Now, we fix $M_1,M_2\in\N$, with $M_1\le M_2$ and show that $\uu_{\varepsilon,M_1}\leq \uu_{\varepsilon,M_2}$. As it has been already shown, $\mathcal V_{\varepsilon,M_1}=\mathcal V_{\varepsilon,M_2}=\widetilde{\mathcal V}_{\varepsilon}$.
Further, since $V_{\varepsilon,M_2}\leq V_{\varepsilon,M_1}$ componentwise and $\uu_{\varepsilon,M_1},\uu_{\varepsilon,M_2}$ have nonnegative components, it follows that
\begin{align}
& \int_{\R^d}\bigg (\sum_{i=1}^m\langle\nabla (u_{\varepsilon,M_1})_i,\nabla v_i\rangle+\langle V_{\varepsilon,M_2}\uu_{\varepsilon,M_1},\bm{v}\rangle\bigg )dx \notag\\
\leq & \int_{\R^d}\bigg (\sum_{i=1}^m\langle\nabla (u_{\varepsilon,M_1})_i,\nabla v_i\rangle+\langle V_{\varepsilon,M_1}\uu_{\varepsilon,M_1},\bm{v}\rangle\bigg )dx \notag \\
= & \int_{\R^d}\langle\f,\bm{v}\rangle dx
\label{dis_u_eps_M_2}
\end{align}
for every $\bm{v}\in C^\infty_c(\R^d,\R^m)$ with nonnegative components. From the definition of $\uu_{\varepsilon,M_2}$ and \eqref{dis_u_eps_M_2} it follows that the function  $\bm{w}_{\varepsilon}=\uu_{\varepsilon,M_1}-\uu_{\varepsilon,M_2}$ satisfies the variational inequality 
\begin{align}
\label{dis_u_eps_m_1-u_eps_m_2}
\int_{\R^d}\bigg (\sum_{i=1}^m\langle\nabla (w_{\varepsilon})_i,\nabla v_i\rangle+\langle V_{\varepsilon,M_2}\bm{w}_{\varepsilon},\bm{v}\rangle\bigg )dx \leq 0        
\end{align}
for every $\bm{v}\in C^\infty_c(\R^d,\R^m)$ with nonnegative components. By density, \eqref{dis_u_eps_m_1-u_eps_m_2} can be extended to every $\bm{v}\in \widetilde{\mathcal V}_\varepsilon$, with nonnegative components. In particular, taking as $\bm{v}$ the function whose $i$-th component is $v_i=(w_{\varepsilon})_i\chi_{\{(w_{\varepsilon})_i>0\}}$ for every $i=1,\ldots,m$, we get
\begin{align*}
\int_{\R^d}\bigg (&\sum_{i=1}^m\|\nabla (w_{\varepsilon})_i\|^2\chi_{\{(w_{\varepsilon})_i>0\}}
 +\sum_{i,j=1}^m (V_{\varepsilon,M_2})_{ij}(w_{\varepsilon})_j(w_{\varepsilon})_i\chi_{\{(w_{\varepsilon})_i>0\}}\bigg )dx \leq 0.     
\end{align*}
Since the first term in the previous inequality is nonnegative, it follows that
\begin{align}
\label{dis_int_w_eps_M_1_2}
\int_{\R^d}\sum_{i,j=1}^m (V_{\varepsilon,M_2})_{ij}(w_{\varepsilon})_j(w_{\varepsilon})_i\chi_{\{(w_{\varepsilon})_i>0\}}dx\leq 0.   
\end{align}
Now, we split the sum under the integral sign as
\begin{align}
& \sum_{i,j=1}^m (V_{\varepsilon,M_2})_{ij}(w_{\varepsilon})_j(w_{\varepsilon})_i\chi_{\{(w_{\varepsilon})_i>0\}} \notag \\
= & \sum_{i,j=1}^m (V_{\varepsilon,M_2})_{ij}(w_{\varepsilon})_j\chi_{\{(w_{\varepsilon})_j>0\}}(w_{\varepsilon})_i\chi_{\{(w_{\varepsilon})_i>0\}} \notag \\
& + \sum_{i,j=1}^m (V_{\varepsilon,M_2})_{ij}(w_{\varepsilon})_j\chi_{\{(w_{\varepsilon})_j\leq0\}}(w_{\varepsilon})_i\chi_{\{(w_{\varepsilon,})_i>0\}} \notag \\
\eqqcolon &\ S_{1}+S_{2},
\label{spezz_somma_w_eps_1_2}
\end{align}
and we separately study the two sums. As far as $S_1$ is concerned, recalling that $v_i=(w_{\varepsilon})_i\chi_{\{(w_{\varepsilon})_i>0\}}$ for every $i=1,\ldots,m$, from the definition of $V_{\varepsilon,M_2}$ we get
\begin{align}
\label{stima_S_1_M}
S_{1}
=\langle V_{\varepsilon,M_2}\bm{v},\bm{v}\rangle\geq \varepsilon\|\bm{v}\|^2.
\end{align}
Next, we consider the term $S_{2}$. Let us notice that in the sum when $i=j$, the corresponding term vanishes, since the product $\chi_{\{(w_{\varepsilon})_i>0\}}\cdot \chi_{\{(w_{\varepsilon})_i\leq 0\}}$ appears. Hence,
\begin{equation}
\label{stima_S_2_M}
S_{2}
=\sum_{i,j=1,i\neq j}^m (V_{\varepsilon,M_2})_{ij}(w_{\varepsilon})_j\chi_{\{(w_{\varepsilon})_j\leq0\}}(w_{\varepsilon})_i\chi_{\{(w_{\varepsilon})_i>0\}}\geq0,
\end{equation}
since 
\begin{align*}
(V_{\varepsilon,M_2})_{ij}\leq0, \quad (w_{\varepsilon})_j\chi_{\{(w_{\varepsilon})_j\leq0\}}\leq 0, \quad (w_{\varepsilon})_i\chi_{\{(w_{\varepsilon})_i>0\}}\geq0
\end{align*}
for every $i,j=1,\ldots,m$ with $i\neq j$. From \eqref{dis_int_w_eps_M_1_2}, \eqref{spezz_somma_w_eps_1_2}, \eqref{stima_S_1_M} and \eqref{stima_S_2_M} we deduce that
\begin{align*}
\varepsilon\int_{\R^d}\sum_{i=1}^m((w_{\varepsilon})_i)^2\chi_{\{(w_{\varepsilon})_i>0\}}dx\leq 0.  
\end{align*}
This gives $(w_{\varepsilon})_i\chi_{\{(w_{\varepsilon})_i>0\}}=0$ a.e. in $\R^d$ for every $i=1,\ldots,m$, which means that $\bm{u}_{\varepsilon,M_1}-\bm{u}_{\varepsilon,M_2}\leq \bm{0}$ componentwise almost everywhere in $\R^d$.

Now, we can apply Lemma \ref{lem:rel_VepsM_Veps} to deduce  that $\langle V_{\varepsilon}\uu_{\varepsilon,M},\uu_{\varepsilon,M}\rangle\le \langle V_{\varepsilon,M}\uu_{\varepsilon,M},\uu_{\varepsilon,M}\rangle$ so that $V_{\varepsilon}^{1/2}\uu_{\varepsilon,M}\in L^2(\R^d,\R^m)$ and
$\|V_{\varepsilon}^{1/2}\uu_{\varepsilon,M}\|_2\le \|V_{\varepsilon,M}^{1/2}\uu_{\varepsilon,M}\|_2$. As a byproduct, we deduce that 
$\uu_{\varepsilon,M}\in{\mathcal V}_{\varepsilon}$ and
\begin{align*}
\int_{\R^d}\bigg (\sum_{i=1}^m\langle\nabla (u_{\varepsilon,M})_i,\nabla(u_{\varepsilon,M})_i\rangle+\langle V_{\varepsilon}\uu_{\varepsilon,M},\uu_{\varepsilon,M}\rangle\bigg )dx\le\int_{\R^d}\langle \f,\uu_{\varepsilon,M}\rangle dx.    
\end{align*}
Hence, the family $(\uu_{\varepsilon,M})_{M\in\N}$ is bounded in ${\mathcal V}_{\varepsilon}$ and we can determine a sequence $(M_n)_{n\in\N}$ diverging to $+\infty$ such that the sequence $(\uu_{\varepsilon,M_n})_{n\in\N}$ weakly converges in ${\mathcal V}_{\varepsilon}$ to some function $\tilde\uu_{\varepsilon}$. 

We now prove that $\tilde\uu_{\varepsilon}=\uu_{\varepsilon}\coloneqq({\mathcal H}+\varepsilon)^{-1}\f$. For this purpose, we begin by observing that, taking $\bm{v}=\uu_{\varepsilon,M}$ in \eqref{1},
we conclude that $\|V_{\varepsilon,M}^{1/2}\uu_{\varepsilon,M}\|_2\le\|\f\|_2$ for every $M>0$. Hence, up to a subsequence, we can assume
that $(V_{\varepsilon,M_n}^{1/2}\uu_{\varepsilon,M_n})_{n\in\N}$ weakly converges to some function $\bm{w}$ in $L^2(\R^d,\R^m)$.

Let us fix a function $\bm{v}\in C^{\infty}_c(\R^d,\R^m)$ and observe that 
($V^{1/2}_{\varepsilon,M_n}\bm{v})_{n\in\N}$ converges to $V^{1/2}_{\varepsilon}\bm{v}$ in $L^2(\R^d,\R^m)$ by the dominated convergence theorem. Indeed, $(V^{1/2}_{\varepsilon,M_n}\bm{v})_{n\in\N}$ pointwise converges to 
$V^{1/2}_{\varepsilon}\bm{v}$. Moreover, 
\begin{align*}
\|V^{1/2}_{\varepsilon,M_n}\bm{v}-V^{1/2}_{\varepsilon}\bm{v}\|^2
\le & 2 (\|V^{1/2}_{\varepsilon,M_n}\bm{v}\|^2+\|V^{1/2}_{\varepsilon}\bm{v}\|^2)\\
=& 2 (\langle V_{\varepsilon,M_n}\bm{v},\bm{v}\rangle+
\langle V_{\varepsilon}\bm{v},\bm{v}\rangle)\\
\le &2(\|V_{\varepsilon,M_n}\|+\|V_{\varepsilon}\|) \|\bm{v}\|^2 \\
\le &4\|V_{\varepsilon}\| \|\bm{v}\|^2
\end{align*}
for every $n\in\N$, and the last side of the previous chain of inequalities belongs to $L^1(\R^d)$.

We can infer that
\begin{align*}
\lim_{n\to +\infty}\int_{\R^d}\langle V_{\varepsilon,M_n}\uu_{\varepsilon,M_n},\bm{v}\rangle dx
=&\lim_{n\to +\infty}\int_{\R^d}\langle V_{\varepsilon,M_n}^{1/2}\uu_{\varepsilon,M_n},V_{\varepsilon,M_n}^{1/2}\bm{v}\rangle dx\\
=&\int_{\R^d}\langle \bm{w},V_{\varepsilon}^{1/2}\bm{v}\rangle dx    
\end{align*}
for every $\bm{v}\in C^{\infty}_c(\R^d,\R^m)$.
Indeed, the sequence $(V_{\varepsilon,M_n}^{1/2}\uu_{\varepsilon,M_n})_{n\in\N}$ is bounded and weakly converges in $L^2(\R^d,\R^m)$ to $\bm{w}$. Hence,
\begin{align*}
&\bigg |\int_{\R^d}\langle V_{\varepsilon,M_n}\bm{u}_{\varepsilon,M_n},\bm{v}\rangle dx-\int_{\R^d}\langle \bm{w},V_{\varepsilon}^{1/2}\bm{v}\rangle dx\bigg |\\
\le & \|V_{\varepsilon,M_n}^{1/2}\bm{u}_{\varepsilon,M_n}\|_2\|V_{\varepsilon,M_n}^{1/2}\bm{v}-V_\varepsilon^{1/2}\bm{v}\|_2
+\bigg |\int_{\R^d}\langle V_{\varepsilon,M_n}^{1/2}\uu_{\varepsilon,M_n}-\bm{w},V^{1/2}_{\varepsilon}\bm{v}\rangle dx\bigg |
\end{align*}
for every $\bm{v}\in C^{\infty}_c(\R^d,\R^m)$. 

Writing \eqref{1}, with $\bm{v}\in C^{\infty}_c(\R^d,\R^m)$, and letting $n$ tend to $\infty$ we obtain that
\begin{align}
\int_{\R^d}\bigg (\sum_{i=1}^m\langle\nabla (\tilde u_{\varepsilon})_i,\nabla v_i\rangle+\langle \bm{w},V_{\varepsilon}^{1/2}\bm{v}\rangle\bigg )dx=\int_{\R^d}\langle\f,\bm{v}\rangle\,dx.  
\label{2a}
\end{align}

To complete the proof, we need to identify $\bm{w}$ with the function $V_{\varepsilon}^{1/2}\tilde\uu_{\varepsilon}$.
For this purpose, we observe that
\begin{align}
\int_{\R^d}\langle V_{\varepsilon}^{1/2}\tilde\uu_{\varepsilon},\bm{v}\rangle dx=\lim_{n\to +\infty}
\int_{\R^d}\langle V_{\varepsilon}^{1/2}\uu_{\varepsilon,M_n},\bm{v}\rangle dx
\label{2}
\end{align}
for every $\bm{v}\in C^{\infty}_c(\R^d,\R^m)$.

Next, we write
\begin{align*}
V_{\varepsilon}^{1/2}\uu_{\varepsilon,M_n}=V_{\varepsilon}^{1/2}V_{\varepsilon,M_n}^{-1/2}V_{\varepsilon,M_n}^{1/2}\uu_{\varepsilon,M_n}
\end{align*}
for every $n\in\N$, so that
\begin{align*}
\langle V_{\varepsilon}^{1/2}\uu_{\varepsilon,M_n},\bm{v}\rangle
=\langle V_{\varepsilon,M_n}^{1/2}\uu_{\varepsilon,M_n},V_{\varepsilon,M_n}^{-1/2}V_{\varepsilon}^{1/2}\bm{v}\rangle
\end{align*}
for every $n\in\N$. Taking Lemma \ref {lem:rel_VepsM_Veps} into account, we can easily show that
\begin{align*} 
\|V_{\varepsilon,M_n}^{-1/2}V_{\varepsilon}^{1/2}\bm{v}\|^2\le \varepsilon^{-1}\|V_{\varepsilon}^{1/2}\bm{v}\|^2
\end{align*}
for every $n\in\N$. Since the right-hand side of the previous inequality defines a function in $L^1(\R^d)$, by the dominated convergence theorem we deduce that
$(V_{\varepsilon,M_n}^{-1/2}V_{\varepsilon}^{1/2}\bm{v})_{n\in\N}$ converges to $\bm{v}$ in 
$L^2(\R^d,\R^m)$. Thus, arguing as above,
we conclude that
\begin{align*}
\lim_{n\to +\infty}\int_{\R^d}\langle V_{\varepsilon}^{1/2}\uu_{\varepsilon,M_n},\bm{v}\rangle dx=&
\lim_{n\to +\infty}
\int_{\R^d}\langle V_{\varepsilon,M_n}^{1/2}\uu_{\varepsilon,M_n},V_{\varepsilon,M_n}^{-1/2}V_{\varepsilon}^{1/2}\bm{v}\rangle dx\notag\\
=&\int_{\R^d}\langle\bm{w},\bm{v}\rangle dx
\end{align*}
for every $\bm{v}\in C^{\infty}_c(\R^d,\R^m)$. Comparing this limit with the one in \eqref{2} we get that $\bm{w}=V_{\varepsilon}^{1/2}\tilde\uu_{\varepsilon}$.

Now, from \eqref{2a} we conclude that
\begin{align*}
\int_{\R^d}\bigg (\sum_{i=1}^m\langle\nabla (\tilde u_{\varepsilon})_i,\nabla v_i\rangle+\langle V_{\varepsilon}\tilde\uu_{\varepsilon},\bm{v}\rangle\bigg )dx=\int_{\R^d}\langle\f,\bm{v}\rangle dx  \end{align*}
for every $\bm{v}\in C^{\infty}_c(\R^d,\R^m)$ and, by density, for every $\bm{v}\in{\mathcal V}_{\varepsilon}$.
By uniqueness, it follows that $\tilde\uu_{\varepsilon}=\uu_{\varepsilon}$. 

Arguing as in the last part of the proof of Lemma \ref{lem:3.1} we infer that $(\|\uu_{\varepsilon,M_n}\|_{{\mathcal V}_\varepsilon})_{n\in\N}$ converges to $\|\uu_{\varepsilon}\|_{{\mathcal V}_\varepsilon}$, which implies that $(\uu_{\varepsilon,M_n})_{n\in\N}$ converges to $\uu_\varepsilon$ in ${\mathcal V}_\varepsilon$. Being the limit of a sequence of componentwise nonnegative functions, also $\uu_\varepsilon$ is a componentwise nonnegative function.

Finally, since $(\uu_{\varepsilon,M})_{n\in\N}$ is a componentwise increasing sequence, we conclude that the whole sequence $(\uu_{\varepsilon,M})_{M\in\N}$ converges to $\uu_\varepsilon$ in ${\mathcal V}_\varepsilon$.

For a general $\f\in L^2_c(\R^d,\R^m)$ we split $\f=\f^+-\f^-$, where $f^+_i=\max\{f_i,0\}$ and $f^-_i=\max\{-f_i,0\}$ for every $i=1,\ldots,m$. Applying the above arguments to $\f^+$ and $\f^-$, we get the assertion in the general case.
\end{proof}
 
Now, we can state and prove the following crucial result.

\begin{theo}
\label{prop:stime_L1}
Let $ \f\in L^{\infty}_{\text{c}}(\R^d,\R^m)$,
with all the components which are nonnegative on $\R^d$ and set $ \uu = \dot{\H}^{-1}\f$. Then, $\ V\uu$ and $\Delta \uu$ belong to $L^1(\R^d,\R^m)$ and there exists a positive constant $C$, independent of $\uu$, such that
\begin{equation}
\label{eq:VH^-1:limitato}
\int_{\R^d}\|V\uu\|dx \le C\int_{\R^d}\|\f\|dx,
\end{equation}
\begin{equation}
\label{eq:DeltaH^-1:limitato}
\int_{\R^d}\|\Delta \uu\|dx \le (1+C)\int_{\R^d}\|\f\|dx.
\end{equation}	
Moreover, $\uu\in W^{1,1}_{\text{loc}}(\R^d,\R^m)$ and for any measurable set $ E\subset\R^d $, with finite measure, it holds that 
\begin{align}
\label{stima-L1-loc-grad}
\int_{E}\|\nabla \uu\|dx\le C(d)\abs{E}^{\frac1d}\int_{\R^d}\|\f\|dx.
\end{align}
Finally, for all compact set $ K\subset\R^d$, it holds that
\begin{equation}
\int_{K}\|\uu\|dx\le C(K,d,\lambda_V)\int_{K}\|\f\|dx.
\label{stima-L1-loc}
\end{equation}
\end{theo}
	
\begin{proof}
Let us fix $\f\in L^{\infty}_{\text{c}}(\R^d,\R^m),$ with nonnegative components, $\varepsilon>0$ and set $\uu_{\eps} = (\H+\eps)^{-1}\f$.

Preliminarily, we show that $\uu_\varepsilon$ converges to $\uu$ almost everywhere in $\R^d$ as $\varepsilon$ tends to zero.
For this purpose, we begin by recalling that
from Remark \ref{rem:3.2} it follows that we can determine a sequence $(\varepsilon_n)_{n\in\N}$ converging to zero such that $(\uu_{\eps_n})$ converges almost everywhere to $\uu$ as $n$ tends to $\infty$. 
Next, we observe that, for every $\varepsilon_1<\varepsilon_2$, the resolvent identity shows that 
\begin{align*}
\uu_{\varepsilon_1}-\uu_{\varepsilon_2}
=&(\varepsilon_1+\H)^{-1}\f-(\varepsilon_2+\H)^{-1}\f\\
=&(\varepsilon_2-\varepsilon_1)(\varepsilon_1+\H)^{-1}(\varepsilon_2+\H)^{-1}\f
\end{align*}
and the right-hand side of this formula is a function which has nonnegative components in $\R^d$, due to Proposition \ref{prop:conv_u_eps_M}.
We conclude that the whole family $(\uu_{\eps})_{\varepsilon>0} $ is decreasing (with respect to $\eps$) and, consequently, it converges almost everywhere in $\R^d$ (to $\uu$). This implies, in particular, that $\uu$ is componentwise nonnegative in $\R^d$.
		
Now, we have all the tools to prove estimate
\eqref{eq:VH^-1:limitato}. For this purpose, we fix $M,N\in\N$ and set $\uu_{\varepsilon,M}={\mathcal H}_{\varepsilon,M}^{-1}\f$ and 
$\uu_{\varepsilon,M,N}={\mathcal H}_{\varepsilon,M,N}^{-1}\f$, where 
${\mathcal H}_{\varepsilon,M}$ (resp.
${\mathcal H}_{\varepsilon,M,N}$)
is the operator associated with $-\Delta+V_{\varepsilon,M}$ (resp. 
$-\Delta+V_{\varepsilon,M,N}$) as defined at the beginning of Section \ref{sect-defH},
and $V_{\varepsilon,M,N}$ is the matrix-valued function defined in Lemma \ref{lemma-V-eps-M-N}.

To improve the readability, we split the rest of the proof of \eqref{eq:VH^-1:limitato} into some steps.

{\em Step 1.} Here, we prove that
\begin{align}
\int_{\R^d}\lambda_{V_{\varepsilon,M,N}}\|\uu_{\varepsilon,M,N}\|dx\leq \int_{\R^d}\|\f\|dx,    
\label{disug-u-eps-M-N}
\end{align}
where, for every $x\in\R^d$, $\lambda_{V_{\varepsilon,M,N}}(x)$ denotes the minimum eigenvalue of the matrix $V_{\varepsilon,M,N}(x)$.

For this purpose, we fix $i\in\{1,\ldots,M\}$ and multiply both sides of the equation
\begin{align*}
-\Delta (u_{\varepsilon,M,N})_i
+(V_{\varepsilon,M,N}\uu_{\varepsilon,M,N})_i=f_i
\end{align*}
by $(u_{\varepsilon,M,N})_i/\|\uu_{\varepsilon,M,N}\|\chi_{\{\uu_{\varepsilon,M,N}\neq\bm{0}\}}$ and sum up $i$ from $1$ to $m$, we get
\begin{align}
&-\frac{\langle\Delta \uu_{\varepsilon,M,N},\uu_{\varepsilon,M,N}\rangle}{\|\uu_{\varepsilon,M,N}\|}\chi_{\{\uu_{\varepsilon,M,N}\neq\bm{0}\}}
+\frac{\langle V_{\varepsilon,M,N}\uu_{\varepsilon,M,N},\uu_{\varepsilon,M,N}\rangle}{\|\uu_{\varepsilon,M,N}\|}\chi_{\{\uu_{\varepsilon,M,N}\neq\bm{0}\}}\notag\\
=&\frac{\langle \f,\uu_{\varepsilon,M,N}\rangle}{\|\uu_{\varepsilon,M,N}\|}\chi_{\{\uu_{\varepsilon,M,N}\neq\bm{0}\}}.   
\label{eq_epsilon_N}
\end{align}

Since the entries of the matrix $V_{\varepsilon,M,N}$ are bounded, by the bounded perturbation theorem for semigroups and classical results (see \cite[Chapter C-II, Examples 1.5, p. 251]{arendt86}), $\uu_{\varepsilon,M,N}$ and
$\Delta\uu_{\varepsilon,M,N}$ belong to $L^1(\R^d,\R^m)$. Hence, from Lemma \ref{lem:kato_vett} it follows that
\begin{align}
\label{kato_vett}
\Delta \|\uu_{\varepsilon,M,N}\|\geq\chi_{\{\uu_{\eps,M,N}\neq\bm{0}\}}\|\uu_{\varepsilon,M,N}\|^{-1}\langle \uu_{\varepsilon,M,N},\Delta\uu_{\varepsilon,M,N}\rangle     
\end{align}
in the sense of distributions. By replacing \eqref{kato_vett} in \eqref{eq_epsilon_N} we deduce that
\begin{align*}
-\Delta\|\uu_{\varepsilon,M,N}\|
+\frac{\langle V_{\varepsilon,M,N}\uu_{\varepsilon,M,N},\uu_{\varepsilon,M,N}\rangle}{\|\uu_{\varepsilon,M,N}\|}\chi_{\{\uu_{\varepsilon,M,N}\neq\bm{0}\}}\leq \|\f\|
\end{align*}
in the sense of distributions. 

Let $(\varphi_n)_{n\in\N}\subset C_c^\infty(\R^d)$ be a sequence of non-negative functions such that $\varphi_n$ converges to $1$ and $\Delta\varphi_n$ converges to $0$ in a dominated way, as $n$ tends to $\infty$. Multiplying both sides of the previous inequality by $\varphi_n$ and integrating over $\R^d$, taking \eqref{kato_vett} into account, gives
\begin{align*}
& -\int_{\R^d}\|\uu_{\varepsilon,M,N}\|\Delta \varphi_n dx
+\int_{\R^d}\frac{\langle V_{\varepsilon,M,N}\uu_{\varepsilon,M,N}\uu_{\varepsilon,M,N}\rangle}{\|\uu_{\varepsilon,M,N}\|}\chi_{\{\uu_{\varepsilon,M,N}\neq\bm{0}\}}\varphi_n dx\\
= & -\langle\Delta\|\uu_{\varepsilon,M,N}\|,\varphi_n\rangle_{{\mathcal D'},\mathcal D}+\int_{\R^d}\frac{\langle V_{\varepsilon,M,N}\uu_{\varepsilon,M,N},\uu_{\varepsilon,M,N}\rangle}{\|\uu_{\varepsilon,M,N}\|}
\chi_{\{\uu_{\varepsilon,M,N}\neq\bm{0}\}}
\varphi_n dx\\
\le & \int_{\R^d}\|\f\|\varphi_n dx.
\end{align*}
By dominated convergence, we can let $n$ tend to $+\infty$ to deduce that
\begin{align*}
\int_{\R^d} \frac{\langle V_{\varepsilon,M,N}\uu_{\varepsilon,M,N},\uu_{\varepsilon,M,N}\rangle}{\|\uu_{\varepsilon,M,N}\|}\chi_{\{\uu_{\varepsilon,M,N}\neq\bm{0}\}}\,dx
\leq \int_{\R^d} \|\f\|dx.   
\end{align*}
From this inequality, estimate \eqref{disug-u-eps-M-N}
follows at once.
  
{\em Step 2.} Here, we prove that $\|\uu_{\varepsilon,M,N}\|\geq \|\uu_{\varepsilon,M}\|$ for every $N\geq N(\varepsilon,M,m)$, where $N(\varepsilon,M,m)$ is defined in Lemma \ref{lemma-V-eps-M-N} to guarantee that $V_{\varepsilon,M,N}\ge\frac{\varepsilon}{2}$ in the sense of forms for every $N\ge N(\varepsilon,M,m)$.

To prove the claim, we observe that $V_{\varepsilon,M}\geq V_{\varepsilon,M,N}$ componentwise for every $N\in\N$. 
Since $\uu_{\varepsilon,M}$ has nonnegative components (see Proposition \ref{prop:conv_u_eps_M}), this implies that
\begin{align}
& \int_{\R^d}\bigg (\sum_{i=1}^m\langle\nabla (u_{\varepsilon,M})_i,\nabla v_i\rangle+\langle V_{\varepsilon,M,N}\uu_{\varepsilon,M},\bm{v}\rangle\bigg )dx \notag\\
\leq & \int_{\R^d}\bigg (\sum_{i=1}^m\langle\nabla (u_{\varepsilon,M})_i,\nabla v_i\rangle+\langle V_{\varepsilon,M}\uu_{\varepsilon,M},\bm{v}\rangle\bigg )dx \notag \\
= & \int_{\R^d}\langle\f,\bm{v}\rangle dx
\label{dis_u_eps_M}
\end{align}
for every $\bm{v}\in C^\infty_c(\R^d,\R^m)$ with nonnegative components. Further, the definition of $\bm{u}_{\varepsilon,M,N}$ gives
\begin{align}
\int_{\R^d}\bigg (\sum_{i=1}^m\langle\nabla (u_{\varepsilon,M,N})_i,\nabla v_i\rangle+\langle V_{\varepsilon,M,N}\uu_{\varepsilon,M,N},\bm{v}\rangle\bigg )dx = & \int_{\R^d}\langle\f,\bm{v}\rangle dx    
\label{ug_u_eps_M_N}
\end{align}
for every $\bm{v}\in C^\infty_c(\R^d,\R^m)$.
If we set $\bm{w}_{\varepsilon,M,N}\coloneqq\uu_{\varepsilon,M}-\uu_{\varepsilon,M,N}$, from \eqref{dis_u_eps_M} and \eqref{ug_u_eps_M_N} it follows that
\begin{align}
\label{dis_u_eps_m_-u_eps_m_n}
\int_{\R^d}\bigg (\sum_{i=1}^m\langle\nabla (w_{\varepsilon,M,N})_i,\nabla v_i\rangle+\langle V_{\varepsilon,M,N}\bm{w}_{\varepsilon,M,N},\bm{v}\rangle\bigg )dx \leq 0 \end{align}
for every $\bm{v}\in C^\infty_c(\R^d,\R^m)$ with nonnegative components. By density, \eqref{dis_u_eps_m_-u_eps_m_n} can be extended to every $\bm{v}\in H^1(\R^d,\R^m)$ with nonnegative components. In particular, taking as $\bm{v}$ the function whose $i$-th component is $v_i=(w_{\varepsilon,M,N})_i\chi_{\{(w_{\varepsilon,M,N})_i>0\}}\in H^1(\R^d)$ for every $i=1,\ldots,m$, we get
\begin{align*}
\int_{\R^d}\bigg (&\sum_{i=1}^m\|\nabla (w_{\varepsilon,M,N})_i\|^2\chi_{\{(w_{\varepsilon,M,N})_i>0\}}\\
& +\sum_{i,j=1}^m (V_{\varepsilon,M,N})_{ij}(w_{\varepsilon,M,N})_j(w_{\varepsilon,M,N})_i\chi_{\{(w_{\varepsilon,M,N})_i>0\}}\bigg )dx \leq 0.     
\end{align*}
Since the first term in the previous inequality is nonnegative, it follows that
\begin{align}
\label{dis_int_w_eps_M_N}
\int_{\R^d}\sum_{i,j=1}^m (V_{\varepsilon,M,N})_{ij}(w_{\varepsilon,M,N})_j(w_{\varepsilon,M,N})_i\chi_{\{(w_{\varepsilon,M,N})_i>0\}}dx\leq 0.   
\end{align}
Now we split the sum under the integral sign as
\begin{align}\label{spezz_somma_w_eps_M_N}
& \sum_{i,j=1}^m (V_{\varepsilon,M,N})_{ij}(w_{\varepsilon,M,N})_j(w_{\varepsilon,M,N})_i\chi_{\{(w_{\varepsilon,M,N})_i>0\}} \notag \\
= & \sum_{i,j=1}^m (V_{\varepsilon,M,N})_{ij}(w_{\varepsilon,M,N})_j\chi_{\{(w_{\varepsilon,M,N})_j>0\}}(w_{\varepsilon,M,N})_i\chi_{\{(w_{\varepsilon,M,N})_i>0\}} \notag \\
& + \sum_{i,j=1}^m (V_{\varepsilon,M,N})_{ij}(w_{\varepsilon,M,N})_j\chi_{\{(w_{\varepsilon,M,N})_j\leq0\}}(w_{\varepsilon,M,N})_i\chi_{\{(w_{\varepsilon,M,N})_i>0\}} \notag \\
\eqqcolon &\ S_{1,N}+S_{2,N},
\end{align}
and we separately study the two sums. As far as $S_{1,N}$ is concerned, recalling that $v_i=(w_{\varepsilon,M,N})_i\chi_{\{(w_{\varepsilon,M,N})_i>0\}}$ for every $i=1,\ldots,m$, from Lemma \ref{lemma-V-eps-M-N} we get
\begin{align}
\label{stima_S_1}
S_{1,N}
=\langle V_{\varepsilon,M,N}\bm{v},\bm{v}\rangle\geq \frac\varepsilon2\|\bm{v}\|^2
\end{align}
for every $N\geq N(\varepsilon,M,m)$. Now we consider $S_{2,N}$. Let us notice that in the sum when $i=j$, the corresponding term vanishes, since the product $\chi_{\{(w_{\varepsilon,M,N})_i>0\}}\cdot \chi_{\{(w_{\varepsilon,M,N})_i\leq 0\}}$ appears. Hence,
\begin{equation}
\label{stima_S_2}
S_{2,N}
=\sum_{i,j=1,i\neq j}^m (V_{\varepsilon,M,N})_{ij}(w_{\varepsilon,M,N})_j\chi_{\{(w_{\varepsilon,M,N})_j\leq0\}}(w_{\varepsilon,M,N})_i\chi_{\{(w_{\varepsilon,M,N})_i>0\}}\geq0,
\end{equation}
since 
\begin{align*}
(V_{\varepsilon,M,N})_{ij}\leq0, \quad (w_{\varepsilon,M,N})_j\chi_{\{(w_{\varepsilon,M,N})_j\leq0\}}\leq 0, \quad (w_{\varepsilon,M,N})_i\chi_{\{(w_{\varepsilon,M,N})_i>0\}}\geq0
\end{align*}
for every $i,j=1,\ldots,m$ with $i\neq j$. From \eqref{dis_int_w_eps_M_N}, \eqref{spezz_somma_w_eps_M_N}, \eqref{stima_S_1} and \eqref{stima_S_2} we deduce that
\begin{align*}
\frac\varepsilon2\int_{\R^d}\sum_{i=1}^m((w_{\varepsilon,M,N})_i)^2\chi_{\{(w_{\varepsilon,M,N})_i>0\}}dx\leq 0    
\end{align*}
for every $N\geq N(\varepsilon,M,m)$. This gives $(w_{\varepsilon,M,N})_i\chi_{\{(w_{\varepsilon,M,N})_i>0\}}=0$ almost everywhere in $\R^d$ for every $i=1,\ldots,m$ and every $N\geq N(\varepsilon,M,m)$, which means that $\bm{u}_{\varepsilon,M}-\bm{u}_{\varepsilon,M,N}\leq \bm{0}$ componentwise almost everywhere in $\R^d$, for every $N\geq N(\varepsilon,M,m)$. Recalling that, by Proposition \ref{prop:conv_u_eps_M}, $\uu_{\varepsilon,M}\geq\bm{0}$ componentwise, we conclude that 
\begin{equation}
\|\uu_{\varepsilon,M}\|\leq \|\uu_{\varepsilon,M,N}\|
\label{stima-1}
\end{equation}
for every $N\geq N(\varepsilon,M,m)$.
  
Since $\lambda_{V_{\varepsilon,M,N}}>0$ for every $N\geq N(\varepsilon,M,m)$, from \eqref{disug-u-eps-M-N} and taking \eqref{stima-1} into account, we deduce that
\begin{align}
\int_{\R^d} \lambda_{V_{\varepsilon,M,N}}\|\uu_{\varepsilon,M}\|dx\leq \int_{\R^d}\|\f\|dx
\label{stima-2}
\end{align}
for every $N\geq N(\varepsilon,M,m)$. For every fixed $x\in\R^d$ and $N$ sufficiently large, the matrix $V_{\varepsilon,M,N}(x)$  equals the matrix $V_{\varepsilon,M}(x)$, and, in particular, for such values of $N$, $\lambda_{V_{\varepsilon,M,N}}(x)=\lambda_{V_{\varepsilon,M}}(x)$. By applying Fatou's Lemma to estimate \eqref{stima-2}, we deduce that 
\begin{align*}
\int_{\R^d} \lambda_{V_{\varepsilon,M}}\|\uu_{\varepsilon,M}\|dx\leq \int_{\R^d}\|\f\|dx.
\end{align*}
Lemma \ref{lem:rel_VepsM_Veps} gives that $\lambda_{V_{\varepsilon,M}}\geq \lambda_{V_\varepsilon}>0$ and hence
\begin{align}\label{eq:stima-3}
\int_{\R^d} \lambda_{V_{\varepsilon}}\|\uu_{\varepsilon,M}\|dx\leq \int_{\R^d}\|\f\|dx.
\end{align}

By Proposition \ref{prop:conv_u_eps_M}, $\uu_{\varepsilon,M}$ converges to $\uu_{\varepsilon}$ in ${\mathcal V}_{\varepsilon}$, so that it converges to $\uu_{\varepsilon}$ in $L^2(\R^d,\R^m)$. In particular, up to a sequence, we can assume that $\uu_{\varepsilon,M}$ converges almost everywhere in $\R^d$ to $\uu_{\varepsilon}$ as $M$ tends to $+\infty$. Applying once again Fatou's Lemma to estimate \eqref{eq:stima-3}, we conclude that
\begin{align}
\int_{\R^d} \lambda_{V_{\varepsilon}}\|\uu_{\varepsilon}\|dx\leq \int_{\R^d}\|\f\|dx
\label{stima-L1-V-eps-u-eps}
\end{align}
for every $\varepsilon>0$.
To conclude, we notice that, since $\uu_\varepsilon$ pointwise converges almost everywhere in $\R^d$ to $\uu$ as $\varepsilon$ tends to $0$ and $0\le\lambda_V\leq \lambda_{V_\varepsilon}$, by Fatou's Lemma we get
		\begin{equation}
			\label{int_lambda_u}
			\int_{\R^d} \lambda_{ V}\|\uu\|dx
			\leq \liminf_{\varepsilon\to0}\int_{\R^d}\lambda_V\|\uu_\varepsilon\|dx
			\leq \liminf_{\varepsilon\to0}\int_{\R^d}\lambda_{V_\varepsilon}\|\uu_\varepsilon\|dx
			\leq \int_{\R^d}\|\f\|dx.
		\end{equation}
On the other hand, From Hypothesis \ref{hyp-2} it follows that
$\|V\uu\|\le \Lambda_V\|\uu\|\le C\lambda_V\|\uu\|$,	which implies that
\begin{align*}
\int_{\R^d} \|V\uu\|\,dx\leq C\int_{\R^d}\| \f\|\,dx. 
\end{align*}
		
Now we consider the term $\Delta \uu$. Since $\uu_\varepsilon$ converges to $\uu$ in $L^1(\Omega,\R^m)$ for every bounded measurable set $\Omega\subset\R^d$ (see Remark \ref{rem:3.2}), $\Delta u_i$ is the limit of $\Delta(u_{\varepsilon})_i$ in the sense of distributions, for every $i=1,\ldots,m$. If $\varphi\in C^\infty_c(\R^d)$, then we get
\begin{align}
\label{espressione_laplaciano}
-\langle\Delta (u_\varepsilon)_i,\varphi\rangle_{{\mathcal D}',{\mathcal D}}
= \int_{\R^d} f_i\varphi\,dx-\int_{\R^d} (V\uu_{\varepsilon})_{i}\varphi\,dx-\int_{\R^d}\varepsilon(u_\varepsilon)_i\varphi\,dx.
\end{align}
Since $\varphi$ has compact support, the last integral in the right-hand side vanishes as $\varepsilon$ tends to $0$. We recall that 
$\|V\uu_\varepsilon-V\uu\|\leq C\lambda_V\|\uu_\varepsilon-\uu\|$
for every $\varepsilon>0$. Since $\uu_\varepsilon$ pointwise converges almost everywhere in $\R^d$ to $\uu$ as $\varepsilon$ tends to $0$, and 
$0\le \uu_\eps \le \uu$, it follows that $\lambda_V\|\uu_\varepsilon-\uu\|
\leq 2\lambda_V\|\uu\|$ and the right-hand side of this inequality belongs to $L^1(\R^d)$ (due to \eqref{int_lambda_u}). By the dominated convergence theorem we conclude that $V\uu_\varepsilon$ converges to $V\uu$ in $L^1(\R^d,\R^m)$ as $\varepsilon$ vanishes. Letting $\varepsilon$ to $0$ in \eqref{espressione_laplaciano} we infer that $-\Delta u_i=f_i-(V\uu)_i$ in the sense of distributions for every $i=1,\ldots,m$. This means that $\Delta u_i\in L^1(\R^d)$ for every $i=1,\ldots,m$, i.e., $\Delta\uu\in L^1(\R^d,\R^m)$ and
$\|\Delta\uu\|_1\leq (1+C)\|\f\|_1$. Formula \eqref{eq:DeltaH^-1:limitato} is so proved.
		
We turn to the gradient estimates. From
\eqref{stima-L1-V-eps-u-eps} and observing that $\Lambda_{V_{\varepsilon}}=\Lambda_V+\varepsilon\le C\lambda_V+\varepsilon\le C(\lambda_V+\varepsilon)=C\lambda_{V_{\varepsilon}}$, we conclude that
$V\uu_{\varepsilon}$, $V_{\varepsilon}\uu_{\varepsilon}$ belong to $L^1(\R^d,\R^m)$ and
\begin{align}
\|V\uu_{\varepsilon}\|_1+
\|V_{\varepsilon}\uu_{\varepsilon}\|_1
\le 2\|\Lambda_{V_{\varepsilon}}\uu_{\varepsilon}\|_1
\le C\|\f\|_1.
\label{stimaaaa}
\end{align}
Hence, by difference, $\Delta \uu_\varepsilon\in L^1(\R^d,\R^m)$ and
\begin{equation}
\label{eq:Delta-eps}
\|\Delta\uu_\eps\|_1\le (1+C)\|\f\|_1.
\end{equation}

From \cite[Appendix A]{benilan-brezis-crandall:1975} 
we infer that, if $E$ is a measurable subset of $\R^d$ with bounded measure and $1\le q<\frac{d}{d-1}$, then
\begin{align}
\label{stima_locale_gradiente}
\int_E|\nabla (u_\varepsilon)_i|^qdx
\leq C(d,q)|E|^{1-\frac{(d-1)q}{d}}\|\Delta (u_\varepsilon)_i\|_1^q,
\end{align}
for every $i=1,\ldots,m$. By combining \eqref{eq:Delta-eps} and \eqref{stima_locale_gradiente}, we infer that
\begin{align}
\int_E|\nabla(u_\varepsilon)_i|^q dx
&\leq C'(d,q)|E|^{1-\frac{(d-1)q}{d}}\|\f\|^q_1
\label{eq:stima:grad:ueps:i}
\end{align}  
for every $i=1,\ldots,m$. Further, if $Q$ is a cube, then, from estimate \eqref{equation2.1}, it follows that
\begin{align*}
\int_Q\left(\|\nabla \uu_\varepsilon\|
+\|V\uu_\varepsilon\|\right)dx
\ge\int_Q\left(\|\nabla \uu_\varepsilon\|
+\frac{\langle V\uu_{\varepsilon},\uu_{\varepsilon}\rangle}{\|\uu_{\varepsilon}\|}\right)dx\ge C(Q,\lambda_V)\int_Q\|\uu_{\varepsilon}\|dx.
\end{align*}
By putting together this last inequality with \eqref{stimaaaa} and \eqref{eq:stima:grad:ueps:i} 
it follows that
\begin{align}
\int_Q\|\uu_\varepsilon\|dx \leq C(Q,d,\lambda_V)\|\f\|_1.
\label{eq:stima:ueps:i}
\end{align}

From estimates \eqref{eq:stima:grad:ueps:i} and \eqref{eq:stima:ueps:i} and the Sobolev embedding theorems we deduce that $\uu_{\varepsilon}\in W^{1,q}_{\rm loc}(\R^d,\R^m)$ for every $\varepsilon>0$ and $q\in\left[1,\displaystyle \frac {d}{d-1}\right)$, and its $W^{1,q}(\Omega,\R^m)$-norm is bounded by a constant independent of $\varepsilon>0$, for every bounded and measurable set $\Omega\subset \R^d$. By taking the weak limit of $\uu_\varepsilon$ in $W^{1,q}(\Omega,\R^m)$ as $\varepsilon$ tends to $0$, we infer that $\uu\in W^{1,q}(\Omega,\R^m)$ for every bounded and measurable $\Omega\subset \R^d$ and every $q\in\left (1,\displaystyle \frac {d}{d-1}\right)$. On the other hand, since $(\uu_\varepsilon)_{\varepsilon>0}$ converges weakly in $W^{1,q}(\Omega,\R^m)$, we deduce that $\|\uu \|_{W^{1,q}(\Omega,\R^m)}\le \liminf_{\varepsilon\to 0}\|\uu_\varepsilon\|_{W^{1,q}(\Omega,\R^m)}$. Thus, \eqref{eq:stima:grad:ueps:i} and Fatou's lemma imply that
$$
\int_E|\nabla u_i|^qdx
\leq C'(d,q)|E|^{1-\frac{(d-1)q}{d}}\|\f\|_1^q
$$
for every $i=1,\ldots,m$ and $q\in \left (1,\displaystyle \frac {d}{d-1}\right)$.

Such estimate holds true also when $q=1$. Indeed, by applying H\"older inequality, we get
\begin{align*}
\int_{E}|\nabla u_i|dx
\leq & |E|^{(q-1)/q}\left(\int_E|\nabla u_i|^qdx\right)^{1/q}
\leq (C'(d,q))^{1/q}|E|^{\frac{1}{d}}\|\f\|_1
\end{align*}
and \eqref{stima-L1-loc-grad} is proved.
Estimate \eqref{stima-L1-loc} follows by letting $\varepsilon$ tend to zero in \eqref{eq:stima:ueps:i} and aking Remark \ref{rem:3.2} into account.
The proof is now complete.
\end{proof}

From Theorem \ref{prop:stime_L1} we deduce the following corollaries.

\begin{cor}
\label{coro-1}
For every $\varepsilon>0$, each operator $({\mathcal H}+\varepsilon)^{-1}$ extends to a bounded operator mapping from $L^1(\R^d,\R^m)$ into itself and 
\begin{align*}
\|\varepsilon(\mathcal H+\varepsilon)^{-1}\bm{f}\|_1\leq \|\bm{f}\|_1.  
\end{align*}
Moreover, for every $\f\in L^1(\R^d,\R^m)$, the function 
$(\mathcal H+\varepsilon)^{-1}\f$ is the unique solution of the equation $-\Delta\uu+V_{\varepsilon}\uu=\f$ in the sense of distributions, which belongs to $L^1(\R^d,\R^m)$, together with its Laplacian and the function $V\uu$.
\end{cor}

\begin{proof}
The first part of the statement follows from a straightforward density argument using formula \eqref{stima-L1-V-eps-u-eps}.

To prove the second part of the statement, we begin by showing the uniqueness part. For this purpose, we fix $\uu\in L^1(\R^d,\R^m)$ such that $V\uu, \Delta\uu\in L^1(\R^d,\R^m)$ and $\Delta\uu -V_{\varepsilon}\uu=\bm{0}$ almost everywhere in $\R^d$.
Then, arguing as in Step 1 of the proof of Theorem \ref{prop:stime_L1}, we can show that $\uu=\bm{0}$. Indeed, scalarly multiplying both sides of the previous equation by $\|\uu\|^{-1}\uu\chi_{\{\uu\neq\bm{0}\}}$ and using Lemma \ref{lem:kato_vett}, we deduce that
\begin{eqnarray*}
\int_{\R^d}(\Delta\|\uu\|)\varphi dx-\int_{\R^d}\|\uu\|^{-1}\langle V_{\varepsilon}\uu,\uu\rangle \chi_{\{\uu\neq\bm{0}\}}\varphi dx\ge 0    
\end{eqnarray*}
for every nonnegative function $\varphi\in C^{\infty}_c(\R^d)$. Now, taking a sequence $(\varphi_n)_{n\in\N}$ of nonnegative smooth functions converging in a dominated way to $1$ in $\R^d$ and with $\Delta\varphi_n$ which converges uniformly in $\R^d$ to zero, we can infer that
\begin{eqnarray*}
\int_{\R^d}\|\uu\|^{-1}\langle V_{\varepsilon}\uu,\uu\rangle \chi_{\{\uu\neq\bm{0}\}}dx\le 0
\end{eqnarray*}
which implies that
$\displaystyle\int_{\R^d}\varepsilon\|\uu\|dx\le 0$,
so that $\uu$ identically vanishes in $\R^d$.

Now, to prove the existence part, we observe that the proof of Theorem \ref{prop:stime_L1}
shows that the operator $({\mathcal H}+\varepsilon)^{-1}$ can be extended to a bounded operator from $L^1(\R^d,\R^m)$ into the set of functions $\uu\in L^1(\R^d,\R^m)$ such that $\Delta\uu, V\uu\in L^1(\R^d,\R^m)$. Let us fix $\f\in L^1(\R^d,\R^m)$ and a sequence $(\f_n)_{n\in\N}$ of smooth and compactly supported functions which converges to $\f$ in $L^1(\R^d,\R^m)$.
The function $\uu_n=({\mathcal H}+\varepsilon)^{-1}\f_n$ solves the equation $-\Delta \uu_n+V_{\varepsilon}\uu_n=\f_n$ for every $n\in\N$. 
Moreover, $\uu_n$ and $-\Delta\uu_n+V_{\varepsilon}\uu_n$ converge, respectively, to $\uu=({\mathcal H}+\varepsilon)^{-1}\f$ and $\f$ in $L^1(\R^d,\R^m)$ as $n$ tends to $+\infty$. Hence, 
\begin{align*}
\int_{\R^d}f_j\varphi dx=&\lim_{n\to +\infty}
\int_{\R^d}(-\Delta\uu_n+V_{\varepsilon}\uu_n)_j\varphi\,dx\\
=&\lim_{n\to +\infty}
\int_{\R^d}(-(u_{n})_j\Delta\varphi+(V_{\varepsilon}\uu_n)_j\varphi)\,dx\\
=&\int_{\R^d}(-u_j\Delta\varphi+(V_{\varepsilon}\uu)_j\varphi)\, dx\\
=&\int_{\R^d}(-\Delta\uu+V_{\varepsilon}\uu)_j\varphi\,dx
\end{align*}
for every $\varphi\in C^{\infty}_c(\R^d)$ and every $j\in\{1,\ldots,m\}$. This yields $-\Delta\uu+V_{\varepsilon}\uu=\f$ as claimed.
\end{proof}
 
\begin{cor}
\label{thm:ext_L1}
The restriction of the operator $\dot{\mathcal H}^{-1}$ to $L^\infty_c(\R^d,\R^m)$ extends to a linear operator $\dot{{\mathcal H}}_{-1}$ from $L^1(\R^d,\R^m)$ to the set of all functions $\uu\in L^1_{\rm loc}(\R^d,\R^m)$ such that $\Delta\uu, V\uu\in L^1(\R^d,\R^m)$.  Moreover, there exists a positive constant $C$ such that
\begin{align*}
\|V\dot{{\mathcal H}}_{-1}\f\|_1\leq C\|\f\|_1,\qquad\;\,
\|\Delta\dot{{\mathcal H}}_{-1}\f\|_1\leq (1+C)\|\f\|_1
\end{align*}
for every $\f\in L^1(\R^d,\R^m)$.
\end{cor}

We can now prove the following uniqueness result for the equation $-\Delta\uu+V\uu=\f\in L^1(\R^d,\R^m)$.

\begin{theo}
\label{thm:uniq_L1}
If $\uu\in C^\infty_c(\R^d,\R^m)$ and we set $\f\coloneqq-\Delta\uu+V\uu\in L^1(\R^d,\R^m)$, then $\uu=\dot{{\mathcal H}}_{-1}\f$.  
\end{theo}

To prove this result, we need some preliminary tools.
To begin with, we recall a result obtained in \cite[Lemma 2.4]{MR18}.
\begin{lemma}
\label{lem:risultati:Maichine}
For every ${\uu}\in H^1(\R^d,\R^m)$, the function  $\norm{\uu}$ belongs to $H^1(\R^d)$ and 
\begin{equation}\label{eq:D:norma:u}
		\D\norm{\uu} = \frac1{\norm{\uu}}\sum_{j=1}^m (u_j\D u_j)\chi_{\{\uu\ne\bf0\}},
	\end{equation}
	\begin{equation}\label{eq:norma:D:norma:u}
		\norm{\D\norm{\uu}}^2 \le\sum_{j=1}^m\norm{\D u_j}^2.
	\end{equation}
\end{lemma}

In the proof of Theorem \ref{thm:uniq_L1}, we will apply a domination argument. More precisely, we will use the following estimate
	\begin{equation}\label{eq:domination:result}
	\norm{e^{-t\mathcal{H}}{\uu}}\le e^{t\Delta}\norm{\uu},\qquad\;\,t>0,
\end{equation}
which holds true for every $\uu\in L^2(\R^d,\R^m)$, where
$(e^{t\Delta})_{t\ge0}$ is the semigroup generated by the realization of the Laplacian in $L^2(\R^d)$.
Such estimate can be obtained by an abstract argument, due to Ouhabaz \cite[Theorem 2.30]{Ouhabaz}, which can applied provided we prove the following result.

\begin{lemma}
The domain $\mathcal{V}$ of the form $a$ associated to the operator $\mathcal{H}$	
satisfies the following properties:
\begin{enumerate}[\rm (i)]
\item
${\uu}\in\mathcal{V}$ implies $\norm{\uu}\in H^1(\R^d)$,
\item
${\uu}\in\mathcal{V},\ f\in H^1(\R^d)$ such that $|f|\le\norm{\uu}$ implies $|f|\sign{\uu}\in\mathcal{V}$.
\end{enumerate}
Moreover, for every $({\uu},f)\in\mathcal{V}\times H^1(\R^d)$ such that $|f|\le\norm{\uu}$ it holds that
\begin{equation}\label{eq:form:inequality}
		a({\uu},|f|\sign{\uu}) \ge b(\norm{\uu},|f|),
\end{equation}
where $\sign{\uu} = \dfrac{\uu}{\norm{\uu}}\chi_{\{\uu\ne \bm{0}\}}$ and $b$ is the form associated to the scalar operator $-\Delta$ in $L^2(\R^d)$, i.e.,
\begin{equation*}
b(u,v)=\int_{\R^d}\langle \D u(x),\D v(x)\rangle\,dx,\qquad u,v\in	H^1(\R^d).
\end{equation*}

\end{lemma}
\begin{proof}
Property (i) follows from Lemma \ref{lem:risultati:Maichine}.
	
Let us now prove property (ii). For this purpose, let ${\uu}\in\mathcal{V}, \ f\in H^1(\R^d)$ be such that $|f|\le\norm{\uu}$ and set 
$\displaystyle\bm{g}_{\varepsilon}=|f|\frac{\uu}{\norm{\uu}+\eps}$ for every $\varepsilon>0$.

Note that $\bm{g}_{\varepsilon}$ converges to $|f|\sign{\uu}$ almost everywhere in $\R^d$ as $\varepsilon$ tends to zero. Moreover, $\|\bm{g}_{\varepsilon}\|\le|f|$ for every $\varepsilon>0$. Hence, $\bm{g}_{\varepsilon}$ converges to $|f|\sign{\uu}$ in $L^2(\R^d,\R^m)$ as $\eps$ tends to $0$.
	
Next, we observe that, for every  $j=1,\dots,m$, the weak gradient of $({g}_{\varepsilon})_j$ is given by
\begin{align*}
\nabla ({g}_{\varepsilon})_j &= \frac{u_j}{\norm{\uu}+\eps}\D|f| + |f|\left(\frac1{\norm{\uu}+\eps}\D u_j - \frac1{(\norm{\uu}+\eps)^2}u_j\D\norm{\uu}\right).
\end{align*}
Clearly, it converges to
$\left( \frac{u_j}{\norm{\uu}}\D|f| + \frac{|f|}{\norm{\uu}}\D u_j - \frac{|f|}{\norm{\uu}^2}u_j\D\norm{\uu}\right)\chi_{\{\uu\ne \bm{0}\}}$ almost everywhere in $\R^d$,
as $\varepsilon$ tends to zero, and it is dominated by
$\norm{\D|f|} + \norm{\D u_j} + \norm{\D\norm{\uu}}$, which belongs to $L^2(\R^d)$. Hence, $\nabla({g}_{\varepsilon})_j$ converges in $L^2(\R^d,\R^d)$.\\
Summing up, we have proved that $|f|\sign{\uu}$ belongs to $H^1(\R^d,\R^m)$.
	
It remains to show that $V^{1/2}|f|\sign{\uu}$ belongs to $L^2(\R^d,\R^m).$ For this purpose, we observe that
\begin{align*}
\|V^{1/2}|f|{\rm sign}(\uu)\|_2^2 &= \int_{\R^d}\langle V|f|\sign{\uu},|f|\sign{\uu}\rangle dx \\
&= \int_{\{\uu\ne \bm{0}\}} \frac{|f|^2}{\norm{\uu}^2}\langle V{\uu},{\uu}\rangle dx\\
&\le\int_{\R^d}\langle V{\uu},{\uu}\rangle dx<\infty.
\end{align*}


To complete the proof, we need to show that estimate \eqref{eq:form:inequality} holds true. Since $V$ is nonnegative and using formula \eqref{eq:D:norma:u}, we can show that
	\begin{align*}
		&\sum_{j=1}^m\langle\D u_j,\D(|f|\sign{\uu})_j\rangle+\langle V{\uu},|f|\sign{\uu}\rangle - \langle\D\norm{\uu},\D|f|\rangle \\
		=&\bigg\langle \frac1{\norm{\uu}}\sum_{j=1}^m(u_j\D u_j)\chi_{\{\uu\ne 0\}}, \D|f|\bigg \rangle
  +\frac{|f|}{\norm{\uu}}\chi_{\{\uu\ne 0\}}\langle V{\uu},{\uu}\rangle - \langle\D\norm{\uu},\D|f|\rangle\\
  &+ \frac{|f|}{\norm{\uu}}\chi_{\{\uu\ne \bm{0}\}} \bigg( \sum_{j=1}^m\langle\D u_j, \D u_j\rangle -\bigg\langle\frac{1}{\norm{\uu}}\sum_{j=1}^m(u_j\D u_j)\chi_{\{\uu\ne \bm{0}\}}, \D\norm{\bf u}\bigg\rangle\bigg)\\
		\ge&\ \langle \D\norm{\uu}, \D|f|\rangle + \frac{|f|}{\norm{\uu}}\chi_{\{\uu\ne 0\}}\bigg (\sum_{j=1}^m\norm{\D u_j}^2 -\langle\D \norm{\uu}, \D\norm{\uu}\rangle\bigg) - \langle\D\norm{\uu},\D|f|\rangle
	\end{align*}
and the last side of the previous chain of inequalities is nonnegative due to \eqref{eq:norma:D:norma:u}.
	Integrating over $\R^d$ this relation we deduce that
	%
$a({\uu},|f|\sign{\uu}) - b(\norm{\uu},|f|)\ge0$ 
and the proof is complete.
\end{proof}

\begin{proof}[Proof of Theorem $\ref{thm:uniq_L1}$]

We fix $\uu \in C^\infty_c(\R^d,\R^m)$ and set $\bm{f}=-\Delta\uu+V\uu\in L^1(\R^d,\R^m)$. Let us notice that $-\Delta \uu+V_\varepsilon\uu=\f+\varepsilon\uu$ for every $\varepsilon>0$, and so from Corollary \ref{coro-1} we infer that $\uu=(\mathcal H+\varepsilon)^{-1}(\f+\varepsilon\uu)$ for every $\varepsilon\in(0,1)$. Hence, 
\begin{align}
\dot{\mathcal H}_{-1}\f-\uu
= & (\dot{\mathcal H}_{-1}\f-\dot{\mathcal H}^{-1}\f_n)+(\dot{\mathcal H}^{-1}\f_n-({\mathcal H}+\varepsilon)^{-1}\f_n)\notag \\
& +(({\mathcal H}+\varepsilon)^{-1}\f_n-({\mathcal H}+\varepsilon)^{-1}\f)-\varepsilon({\mathcal H}+\varepsilon)^{-1}\uu  
\label{stima-2-righe}
\end{align}
for every $\varepsilon>0$ and every $n\in\N$, where the sequence $(\f_n)_{n\in\N}\subset C^\infty_c(\R^d,\R^m)$ satisfies $\|\f-\f_n\|_1\leq n^{-1}$ for every $n\in\N$. 

Now we fix a bounded and measurable set $\Omega\subset\R^d$. From \eqref{stima-L1-loc} and \eqref{eq:stima:ueps:i}, which can be extended by a density argument to any function $\f\in L^1(\R^d,\R^m)$, we infer that 
\begin{align*}
\|\dot{\mathcal H}_{-1}\f-\dot{\mathcal H}^{-1}\f_n\|_{L^1(\Omega,\R^m)}+\|({\mathcal H}+\varepsilon)^{-1}\f-({\mathcal H}+\varepsilon)^{-1}\f_n\|_{L^1(\Omega,\R^m)}\leq \frac{2C(\Omega,d,\lambda_V)}{n}
\end{align*}
for every $n\in\N$.
Therefore, from \eqref{stima-2-righe} we deduce that
\begin{align}
\|\dot{\mathcal H}_{-1}\f-\uu\|_{L^1(\Omega,\R^m)}
\le\ & 2\frac{C(\Omega,d,\lambda_V)}{n}+
\|\dot{\mathcal H}^{-1}\f_n-({\mathcal H}+\varepsilon)^{-1}\f_n\|_{L^1(\Omega,\R^m)} \notag\\
& +\varepsilon\|({\mathcal H}+\varepsilon)^{-1}\uu\|_{L^1(\Omega,\R^m)}
\label{stima-qualcosa-1}
\end{align}
for every $n\in\N$ and $\varepsilon>0$.

Next, we observe that, using \eqref{eq:domination:result}, we can estimate
\begin{align*}
\eps \norm{(\eps+\mathcal{H})^{-1}{\uu}}
\le \eps\int_0^{+\infty}e^{-s\eps}\|e^{-s\mathcal{H}}\uu\|ds
\le\eps\int_0^{+\infty}e^{-s\eps}e^{s\Delta}\|\uu\|ds=\eps(\varepsilon-\Delta)^{-1}\|\uu\|.
	\end{align*}
Note that
\begin{align*}
(\eps(\varepsilon-\Delta)^{-1}\|\uu\|)(x)=&\varepsilon\int_0^\infty e^{-\varepsilon t}\frac{1}{(4\pi t)^{d/2}}dt\int_{\R^d}\exp\left(-\frac{|x-y|^2}{4t}\right)\|\uu(y)\|dy\\
=& \varepsilon\int_{\frac{1}{\sqrt{\varepsilon}}}^\infty e^{-\varepsilon t}\frac{1}{(4\pi t)^{d/2}}dt\int_{\R^d}\exp\left(-\frac{|x-y|^2}{4t}\right)\|\uu(y)\|dy\\
&+\varepsilon\int_0^{\frac{1}{\sqrt{\varepsilon}}} e^{-\varepsilon t}\frac{1}{(4\pi t)^{d/2}}dt\int_{\R^d}\exp\left(-\frac{|x-y|^2}{4t}\right)\|\uu(y)\|dy=:A_{\varepsilon}(x)+B_{\varepsilon}(x)
\end{align*}
for every $x\in\R^d$.
Since
	\begin{align*}
		A_{\varepsilon}(x)
		\leq \frac{\sqrt{\varepsilon^{d/2}}\|\uu\|_{1}}{(4\pi)^{d/2}}\varepsilon\int_{\frac{1}{\sqrt\varepsilon}}^\infty e^{-\varepsilon t}dt
		\leq \frac{\sqrt{\varepsilon^{d/2}}\|\uu\|_{1}}{(4\pi)^{d/2}},
	\end{align*}
it follows that $A_\varepsilon(x)$ vanishes as $\varepsilon$ tends to $0$ for every $x\in\R^d$. In particular, $A_{\varepsilon}$ vanishes in $L^1(\Omega)$, as $\varepsilon\to 0$, for every bounded measurable set $\Omega\subset\R^d$.
 
As far as $B_\varepsilon$ is concerned, if we integrate with respect to $x\in\R^d$ we get
\begin{align*}
\int_{\R^d}B_\varepsilon(x)dx
 &= \varepsilon\int_0^{\frac{1}{\sqrt{\varepsilon}}}e^{-\varepsilon t}\frac{1}{(4\pi t)^{d/2}}dt\int_{\R^d}dx\int_{\R^d}\exp\left(-\frac{|x-y|^2}{4t}\right)\|\uu(y)\|dy \\
 &= \|\uu\|_1\varepsilon\int_0^{\frac{1}{\sqrt{\varepsilon}}}e^{-\varepsilon t}dt
= (1-e^{-\sqrt{\varepsilon}})\|\uu\|_1.
\end{align*}
	This implies that $B_{\varepsilon}$ tends to $0$ in $L^1(\R^d)$ as $\varepsilon$ tends to $0$.
 
Summing up, we have proved that $\eps \norm{(\eps+\mathcal{H})^{-1}{\uu}}$ vanishes in $L^1(\Omega,\R^m)$ as $\varepsilon$ tends to $0$, for every
bounded and measurable set $\Omega\subset\R^d$.

Finally, we recall that Remark \ref{rem:3.2} shows that $\|\dot{\mathcal H}^{-1}\f_n-({\mathcal H}+\varepsilon)^{-1}\f_n\|_{L^1(\Omega,\R^m)}$ vanishes as $\varepsilon$ tends to zero for every fixed $n\in\N$. Hence, taking the limsup as $\varepsilon$ tends to $0$ in \eqref{stima-qualcosa-1} and, then, the limit as $n$ tends to $+\infty$, we conclude that $\|\uu-\dot{\mathcal H}_{-1}\f\|_{L^1(\Omega,\R^m)}=0$. The arbitrariness of $\Omega$ yields that $\uu=\dot{\mathcal H}_{-1}\f$ almost everywhere in $\R^d$.
\end{proof}

Finally, we can prove the maximal inequality in $L^1(\R^d,\R^m)$.
\begin{theo}
There exists a positive constant $c$ such that for every $\uu\in C^\infty_c(\R^d,\R^m)$ it holds that
\begin{align}
\|\Delta\uu\|_{1}+\|V\uu\|_{1}\leq c\|-\Delta \uu+V\uu\|_{1}.     
\label{coltellate}
\end{align}
\end{theo}
\begin{proof}
We fix $\uu\in C^\infty_c(\R^d,\R^m)$ and set $\f=-\Delta \uu+V\uu\in L^1(\R^d,\R^m)$. From Theorem \ref{thm:uniq_L1} it follows that $\uu=\dot{\mathcal H}_{-1}\f$ and Corollary \ref{thm:ext_L1} yields the assertion with $c=\displaystyle 1+2C$.    
\end{proof}

\section{The \texorpdfstring{$L^p$}{Lp} estimates}
\label{sec:Lp}
We aim to apply the following version of a known theorem (see \cite[Theorem 3.1 \& Remark 3.6]{aus-mar}) with $w=1$. Here, $M$ is the uncentered Hardy-Littlewood maximal function over cubes of $\R^d$,
i.e.,
\begin{eqnarray*}
(Mf)(x)=\sup_{Q\ni x}\frac{1}{|Q|}\int_Q|f|dy,\qquad\;\,x\in\R^d,
\end{eqnarray*}
for every measurable function $f:\R^d\to\R$ and $Q$, as usual, denotes a cube in $\R^d$.

\begin{theo}
\label{2thm:auscher-martell}
Fix $q,s\in (1,\infty)$ and $a\in [1,+\infty)$. Then, there exists a positive constant $C=C(q,d,a,s)$ with the following property. Assume that $F,G,H_1$ and $H_2$ are nonnegative and measurable functions on $\R^d$, with $F\in L^1(\R^d)$, such that for every cube $Q$ there exist non-negative functions $G_Q$ and $H_Q$ with $F(x)\leq G_Q(x)+H_Q(x)$ for almost every $x\in Q$ and
\begin{align}
\label{2stima_ausch_marte_1}
\left(\frac{1}{|Q|}\int_Q H_Q^q(y)\,dy\right)^{1/q}
&\leq a[(M(F))(x)+(M(H_1))(x)+H_2(\bar{x})],\\[8pt]
\label{2stima_ausch-marte_2}
\frac{1}{|Q|}\int_QG_Q(y)\,dy\ &\leq G(x),
\end{align}
for all $x,\bar{x}\in Q$. Then, for all $1\le r\leq q/s$ we have
\begin{align*}
\|M(F)\|_{L^r(\R^d)}
\leq C(\|G\|_{L^r(\R^d)}+\|M(H_1)\|_{L^r(\R^d)}+\|H_2\|_{L^r(\R^d)}).
\end{align*}
\end{theo}
Now we state the result we are interested in.
\begin{theo}\label{thm:auscher-martell:vett}
Let $1\leq p_0<q_0\leq \infty$. Assume that $T$ is a bounded sublinear operator on $L^{p_0}(\R^d,\R^m)$ and there exist constants $\alpha_2>\alpha_1>1$ and $C>0$ such that
\begin{equation}
\left(\frac{1}{|Q|}\int_Q\norm{T\f}^{q_0}dx\right)^{1/q_0}
\leq C\bigg\{\left(\frac{1}{|\alpha_1Q|}\int_{\alpha_1Q}\norm{T\f}^{p_0}dx\right)^{1/p_0}+\norm{(S|\f|)(\bar{x})}\bigg\}
\label{2napoli-milan-0-4}
\end{equation}
for every cube $Q,$ all $\bar{x}\in Q$ and every $\f\in L^\infty(\R^d,\R^m)$ with compact support in $\R^d\setminus\alpha_2Q$, where $S$ is a vector-valued linear operator such that $S(|\bm{g}|)\leq S(|\bm{h}|)$ componentwise if $|\bm {g}|\leq |\bm{h}|$ componentwise. If $S$ is bounded on $L^p(\R^d,\R^m)$, for some
$p\in (p_0,q_0)$, then there exists a positive constant $K$ such that
\begin{align*}
\|T\f\|_{p}
\leq K\|\f\|_{p},\qquad\;\,\f\in L^\infty_c(\R^d,\R^m).
\end{align*}
\end{theo}

\begin{proof}
Along the proof, $c$ denotes a positive constant greater than $1$ which may vary from line to line.
	 
We fix $\f\in L^\infty_{\text{c}}(\R^d,\R^m)$, a cube $Q$ and $x\in Q$. Since $T$ is bounded on $L^{p_0}(\R^d,\R^m)$ we obtain
\begin{align*}
\left(\frac{1}{|\alpha_1Q|}\int_{\alpha_1Q}\norm{T(\chi_{\alpha_2Q}\f)}^{p_0}dy\right)^{1/p_0}
&\leq c\left(\frac{1}{|\alpha_1Q|}\int_{\R^d}\norm{\chi_{\alpha_2Q}\f}^{p_0}dy\right)^{1/p_0}\\
&= c\left(\frac{1}{|\alpha_1Q|}\int_{\alpha_2Q}\norm{\f}^{p_0}dy\right)^{1/p_0}\\
&\leq c\,(M(\norm{\f}^{p_0})(x))^{{1}/{p_0}},
\end{align*}
and since $\alpha_1>1$, one has
\begin{equation}\label{2stima_Lp_1}
\left(\frac{1}{|Q|}\int_{Q}\norm{T(\chi_{\alpha_2Q}\f)}^{p_0}dy\right)^{1/p_0}\leq c\,(M(\norm{\f}^{p_0})(x))^{\frac{1}{p_0}}.
\end{equation}
Further, since the support of $(1-\chi_{\alpha_2Q})\f$ is contained in $\R^d\setminus\alpha_2 Q$, from \eqref{2napoli-milan-0-4} it follows that
\begin{align}
\left(\frac{1}{|Q|}\int_Q\|T((1-\chi_{\alpha_2Q})\f)\|^{q_0}dy\right)^{1/q_0}
\leq & C\Bigg\{\bigg(\frac{1}{|\alpha_1Q|}\int_{\alpha_1Q}\|T((1-\chi_{\alpha_2Q})\f)\|^{p_0}dy\bigg)^{1/p_0} \notag \\
& +\norm{(S|(1-\chi_{\alpha_2Q})\f|)(\bar{x})}\Bigg\}
\label{2stima_Lp_2}
\end{align}
for all $\bar{x}\in Q.$
		
The assumptions on $S$ imply that 
$S(|(1-\chi_{\alpha_2Q})\f|)\le S(|\f|)$, from which it follows that
\begin{equation}\label{eq:stima_con_S}
\norm{S(|(1-\chi_{\alpha_2Q})\f|)}\le\norm{S(|\f|)}.
\end{equation}
The sublinearity of $T$ implies that $\|T((1-\chi_{\alpha_2Q})\f)\|\leq \|T\f\|+\|T\chi_{\alpha_2Q}\f\|$, and so from \eqref{2stima_Lp_1}, \eqref{2stima_Lp_2} and \eqref{eq:stima_con_S} we infer that
\begin{align}
\label{2stima_Lp_3}
\left(\frac{1}{|Q|}\int_Q\|T((1-\chi_{\alpha_2Q})\f)\|^{q_0}dy\right)^{\frac{1}{q_0}}\!\!\!
\leq c\left[(M(\|T\f\|^{p_0})(x))^{\frac{1}{p_0}}+(M(\|\f\|^{p_0})(x))^{\frac{1}{p_0}} + \norm{(S|\f|)(\bar{x})}\right].
\end{align}
Let us fix $p\in(p_0,q_0)$.
We aim to apply Theorem \ref{2thm:auscher-martell} with $F=\|T\f\|^{p_0}\in L^1(\R^d)$, $G=2^{p_0-1}c^{p_0}M(\|\f\|^{p_0}),$ $H_1=\|\f\|^{p_0}$ and $H_2 =\norm{(S|\f|)}^{p_0}.$ 

From the definition of $F$ it follows that
\begin{align*}
|F(x)|\leq 2^{p_0-1}\|(T(\chi_{\alpha_2Q}\f))(x)\|^{p_0}+2^{p_0-1}\|(T((1-\chi_{\alpha_2Q})\f))(x)\|^{p_0}=:G_Q(x)+H_Q(x)
\end{align*}
for almost every $x\in Q$. Arguing as for  \eqref{2stima_Lp_3} with $q=q_0/p_0>1$ we get
\begin{align*}
\left(\frac1{|Q|}\int_{Q}H_Q^{q}(y)\,dy\right)^{\frac1{q}} &= 2^{p_0-1}\left\{\left(\frac1{|Q|}\int_{Q}\|T((1-\chi_{\alpha_2Q})\f)\|^{q_0}dy\right)^{\frac1{q_0}}\right\}^{p_0}\\
&\le c\left[(M(\|T\f\|^{p_0})(x))^{1/p_0}+(M(\|\f\|^{p_0})(x))^{1/p_0} + \norm{(S|\f|)(\bar{x})}\right]^{p_0}\\
&\le c\left[(M(\|T\f\|^{p_0}))(x)+(M(\|\f\|^{p_0}))(x) + \norm{(S|\f|)(\bar{x})}^{p_0}\right]\\
&= c\left[(M(F))(x) + (M(H_1))(x) + H_2(\bar{x})\right]
\end{align*}
and then \eqref{2stima_ausch_marte_1} is verified, while from \eqref{2stima_Lp_1} we infer
\begin{align*}
\frac1{|Q|}\int_{Q}G_Q(y)dy = \frac{2^{p_0-1}}{|Q|}\int_{Q}\|T(\chi_{\alpha_2Q}\f)\|^{p_0}dy\le 2^{p_0-1}c^{p_0}(M(\|\f\|^{p_0}))(x) = G(x),
\end{align*}
i.e., \eqref{2stima_ausch-marte_2} holds. Theorem \ref{2thm:auscher-martell}, with $r=p/p_0$, $q=q_0/p_0$ and $s=q_0/p$, gives that $\exists\ \widetilde{C}>0$ such that
\begin{align}
\notag 
\|M(F)\|_{L^{p/p_0}(\R^d)} &\leq\widetilde{C}\left[\|G\|_{L^{p/p_0}(\R^d)} + \|M(H_1)\|_{L^{p/p_0}(\R^d)}+\|H_2\|_{L^{p/p_0}(\R^d)}\right] \\
&=\widetilde C(2^{p_0-1}c^{p_0}+1)\|M(\|\f\|^{p_0})\|_{L^{p/p_0}(\R^d)} + \widetilde{C} \norm{S|\f|}^{p_0}_p.
\label{app_aus_martell}
\end{align}
Applying the Lebesgue differentiation theorem, we deduce that
\begin{align}
\label{2prop_funz_max}
\|T\f\|_p^{p_0}
= \|F\|_{L^{p/p_0}(\R^d)}\leq \|M(F)\|_{L^{p/p_0}(\R^d)}.
\end{align}
Combining \eqref{app_aus_martell} with \eqref{2prop_funz_max}, the boundedness of the maximal function $M$ and of the operator $S$ on $L^p(\R^d,\mathbb R^m)$ implies that
\begin{align*}
\|T\f\|_p^{p_0}&\le \widetilde C(2^{p_0-1}c^{p_0}+1)\|M(\|\f\|^{p_0})\|_{L^{p/p_0}(\R^d)} + \widetilde C \norm{S|{\bm f|}}^{p_0}_p
\le C_*\norm{\f}_p^{p_0}
\end{align*}
for some positive constant $C_*$ and the assertion follows.
\end{proof}
	
From now on, we denote by $\Omega$ an open set of $\R^d$.

\begin{lemma}\label{lem:subharmonic}
Assume Hypotheses $\ref{hyp-1}$ and $\ref{hyp-2}$. Let $\uu$ be a weak solution of $-\Delta\uu+V\uu =\bm{0}$ in $\Omega$, i.e., $\uu\in L^2_{\rm loc}(\Omega,\R^m)$ with $V^{\frac12}\uu$,  $D_j\uu\in L^2_{\rm loc}(\Omega,\R^m)$ $(j=1,\ldots,d)$  and the equation holds in the sense of distributions. Then, the following properties are satisfied:
\begin{enumerate}[\rm (i)]
\item 
the function $\|\uu\|^2$ is subharmonic on $\Omega$, i.e.,  $\Delta\|\uu\|^2\ge0$ in $\Omega$ in the sense of distributions;
\item 
if $\lambda_{ V}\in B_r$ for some $ r\in(1,\infty)\cup\{\infty\}$, then for every $1<\mu\le2$ and every cube $Q\subset\R^d$ with $\overline{2Q}\subset\Omega$, there exists a non-negative constant $\widetilde C$ such that
\begin{align*}
\left(\operatorname{av}_Q\|V\uu\|^r\right)^{\frac1r}\le\widetilde{C}\operatorname{av}_{\mu Q}(\|V\uu\|).
\end{align*}     
\end{enumerate}
\end{lemma}
\begin{proof}
(i) If $\uu$ is a weak solution of $-\Delta\bm{v} + V\bm{v} = \bm{0}$,  then $\Delta u_i=(V\uu)_i$ in the sense of distributions, for all $i=1,\dots,m.$ For all $\varphi\in C_c^{\infty}(\Omega)$ we get
\begin{align*}
\int_{\Omega}\Delta(|u_i|^2)\varphi dx &= -\int_{\Omega}\D (|u_i|^2)\D\varphi dx= -2\int_{\Omega}u_i\D u_i \D\varphi dx\\
&=-2\int_{\Omega}\D u_i\D(u_i\varphi)dx +2\int_{\Omega}\D u_i\,\D u_i\,\varphi dx\\
&=2\int_{\Omega}\Delta u_i(u_i\varphi)dx +2\int_{\Omega}|\D u_i|^2\varphi dx\\
&=2\int_{\Omega} (Vu)_i\,u_i\,\varphi +2\int_{\Omega}|\D u_i|^2\varphi dx.
\end{align*}
Summing over $i=1,\dots,m$ we infer that
\begin{align*}			 
\int_{\Omega}\Delta(\|\uu\|^2)\,\varphi dx = 2\int_{\Omega}\langle V\uu,\uu\rangle\,\varphi dx + 2\int_{\Omega}\sum_{i=1}^m|\D u_i|^2\varphi dx.  
\end{align*}
In conclusion, since $V$ is non-negative definite, we get
\begin{align*}
\Delta\|\uu\|^2 =2\left(\langle V\uu,\uu\rangle + \sum_{i=1}^m|\D u_i|^2\right) \ge 0,
\end{align*}
and $\|\uu\|\in L^2_{\rm loc}(\Omega)$. This proves that $\|\uu\|^2$ is subharmonic on $\Omega$. 
 
(ii) Let $Q$ be a cube of $\R^d$ such that $\overline{2Q}\subset\Omega.$
First, we observe that $\|\uu\|^2$ is a subharmonic function on a neighborhood of $\overline{2Q}$ by property (i). Since $\lambda_V\in B_r$, we can apply \cite[Corollary 5.3]{auscher-benali:2007} with weight $\omega=\lambda_V$, $s=\frac{1}{2}$ and $\ f=\|\uu\|^2$ as subharmonic function. It follows that for all $\mu\in (1,2]$ there exists $M\ge0$ such that
\begin{equation}\label{eq:risultato:scalare}
\left(\operatorname{av}_Q(\lambda_V\|\uu\|)^r\right)^{\frac1r}\le M\operatorname{av}_{\mu Q}(\lambda_V\|\uu\|).
\end{equation}
Since $V$ is symmetric and nonnegative, it holds that $\|V\|=\max_{\|\bm{v}\|=1}\langle V\bm{v},\bm{v}\rangle$. Hence,
\begin{equation}\label{eq:medie_V_L}
\left(\operatorname{av}_Q(\|V\uu\|)^r\right)^{\frac1r}\le C\left(\operatorname{av}_Q(\lambda_V\|\uu\|)^r\right)^{\frac1r}, 
\end{equation}
where we have taken into account that the eigenvalues of $V$ are comparable, and $C$ is the constant that appears in the Hypothesis \ref{hyp-2}. Finally, we observe that
\begin{equation}
\label{eq:medie_autovalori}
\operatorname{av}_{\mu Q}(\lambda_V\|\uu\|)\le \operatorname{av}_{\mu Q}(\|V\uu\|)
\end{equation}
for all $\mu\in (1,2]$.
Combining \eqref{eq:risultato:scalare}, \eqref{eq:medie_V_L} and \eqref{eq:medie_autovalori},
we obtain the  assertion with $\widetilde{C} =CM$.
\end{proof}
	
\begin{prop}\label{thm:ext_Lp}
Let $V$ be a matrix-valued operator satisfying Hypotheses $\ref{hyp-1}$ and $\ref{hyp-2}$. Further, suppose that $\lambda_{ V}\in B_q$ for some $ q\in(1,\infty)\cup\{\infty\}$. Then, $V\dot{\mathcal H}_{-1}$ and $\Delta\dot{\mathcal H}_{-1}$, defined on $L^1(\R^d,\R^m)$ by Corollary $\ref{thm:ext_L1}$, extend to bounded operators on $L^p(\R^d,\R^m)$ for $1<p<r$, where $r$ is a suitable real number larger than $q$, if $q<\infty$, and $r=\infty$, otherwise.
\end{prop}

\begin{proof}
Let us start with $V\dot{\mathcal H}_{-1}$. By Corollary \ref{thm:ext_L1} it is a linear bounded operator on $L^1(\R^d,\R^m).$ Since $\lambda_{ V}\in B_q$, from the self-improvement of the reverse H\"{o}lder class (see  \cite{gehring}), it follows that if $q$ is real, then $\lambda_{ V}\in B_r$ for some real number $r$ larger than $q$.
		
Fix now a cube $ Q\subset\R^d$ and $\f\in L^{\infty}(\R^d,\R^m)  $ with compact support contained in $ \R^d\setminus 4Q. $ The function $ \uu = \dot{\mathcal H}_{-1}\f\in\dot{\mathcal{V}} $ is well defined and it is a weak solution of $-\Delta\uu + V\uu = \bm{0}$ in $4Q$. By Lemma \ref{lem:subharmonic}(ii), applied with $\Omega=4Q$, we have for $\mu=2$ that
\begin{align*}
\left(\operatorname{av}_Q\|V\dot{\mathcal H}_{-1}\f\|^r\right)^{\frac1r}\le\widetilde{C}\operatorname{av}_{2 Q}\|V\dot{\mathcal H}_{-1}\f\|.
\end{align*} 
Then, we get \eqref{2napoli-milan-0-4} with $ S=0,\,p_0=1,\ q_0=r,\ \alpha_1=2,\ \alpha_2=4,$ and $T= V\dot{\mathcal H}_{-1}$. We can apply Theorem \ref{thm:auscher-martell:vett} and obtain that for all $1<p<r$ there exists $K>0$ such that
\begin{align*}
\|V\dot{\mathcal H}_{-1}\f\|_p \le K\|\f\|_p,\qquad\;\,\bm{f}\in L^{\infty}_c(\R^d,\R^m).
\end{align*}
This implies that $V\dot{\mathcal H}_{-1}$ extends to a linear bounded operator on $L^p(\R^d,\R^m) $ and, by difference, also $ \Delta\dot{\mathcal H}_{-1}$. Indeed, $\|\Delta\dot{\mathcal H}_{-1}\f\|_p\le\|V\dot{\mathcal H}_{-1}\f\|_p+\|\f\|_p\le (K+1)\|\f\|_p$.
\end{proof}

Now, we are ready to prove $L^p$-maximal estimates.

\begin{theo}
\label{thm:maximal_ineq}
Let Hypotheses $\ref{hyp-1}$ and $\ref{hyp-2}$ be satisfied. Further, suppose that $\lambda_V\in B_q$ for some $ q\in(1,\infty)\cup\{\infty\}$. If $q\in (1,\infty)$, then there exists $r$, larger than $q$, and depending only on $V$, such that
\begin{equation}\label{eq:maximal_ineq:Lp}
\norm{\Delta \uu}_p + \norm{V\uu}_p \le C_p\norm{(-\Delta+V) \uu}_p,\qquad\;\,\uu\in C^{\infty}_c(\R^d,\R^m),
\end{equation}
for every $p\in (1,r)$ and some positive constant $C_p$, independent of $\uu$. If $q=\infty$, then estimate \eqref{eq:maximal_ineq:Lp} holds true for every $p\in (1,\infty)$.
\end{theo}

\begin{proof}
Fix $\uu\in C_c^{\infty}(\R^d,\R^m)$ and set $ \f\coloneqq(-\Delta + V)\uu.$ By Theorem \ref{thm:uniq_L1} we know that $ \uu = \dot{\mathcal{H}}_{-1}\f$ and, by Proposition \ref{thm:ext_Lp}, we can  extend $ V\dot{\mathcal{H}}_{-1}$ and $ \Delta\dot{\mathcal{H}}_{-1}$ to linear bounded operators in $L^p(\R^d,\R^m)$ for $ 1<p<r=q+\varepsilon$ for some $\varepsilon>0$. It thus follows that
\begin{align*}
\|\Delta \uu\|_p + \norm{V\uu}_p = \|\Delta\dot{\mathcal H}_{-1} \f\|_p + \|V\dot{\mathcal H}_{-1}\f\|_p
\le (2K+1)\norm\f_p = C_p\norm{(-\Delta + V)\uu}_p, 
\end{align*}
where $K$ is the constant appearing in the proof of Proposition \ref{thm:ext_Lp}.
\end{proof}

\section{Generation results}
\label{sec:Gen}
Throughout this section, we assume Hypotheses \ref{hyp-1} and \ref{hyp-2}.

\noindent Let us introduce the set $X=L^1(\R^d,\R^m)+L^\infty(\R^d,\R^m)$, endowed with the natural norm
\begin{align*}
\|\f\|_X\!=\!\inf\{\|\f_1\|_{L^1(\R^d,\R^m)}+\|\f_2\|_{L^\infty(\R^d,\R^m)}: \f_1\in L^1(\R^d,\R^m), \ \f_2\in L^\infty(\R^d,\R^m), \ \f=\f_1+\f_2 \},  
\end{align*}
and the realization in $X$ of the operator $-\Delta+V$, with
Kato maximal domain, i.e., the operator $T$, defined as follows:
\begin{align*}
 D(T) &=\{\f\in X:V\f\in L^1_{\rm loc}(\R^d,\R^m),\ (\Delta-V)\f\in X\}, \\
 T\f  &=(-\Delta+V)\f, \qquad\;\, \f\in D(T).
\end{align*}
Moreover, for every $p\in[1,\infty]$, we introduce also the part of $T$ in $L^p(\R^d,\R^m)$, i.e., the operator $T_p$ defined by 
\begin{align*}
 D(T_p) &=\{\f\in L^p(\R^d,\R^m):V\f\in L^1_{\rm loc}(\R^d,\R^m),\ (\Delta-V)\f\in L^p(\R^d,\R^m)\}, \nonumber \\
 T_p\f &=(-\Delta+V)\f, \qquad \f\in D(T_p).    
\end{align*}
We notice that $L^p(\R^d,\R^m)\subset X$ and, if $\f\in L^p(\R^d,\R^m)$ then $\|\f\|\in L^1(\R^d)+L^\infty(\R^d)$. Further, for every measurable set $\Omega\subset\R^d$, with finite measure, and every $\f\in X$, it holds that
$\f\in L^1(\Omega,\R^m)$ and there exists a positive constant $C=C(\Omega)$ such that
$\|\f\|_{L^1(\Omega,\R^m)}\le C\|\f\|_X$. 

The following result is the vector-valued counterpart of a well celebrated theorem by Kato (see \cite[Theorem 3]{kato86}).
\begin{theo}
\label{the:generation}
The operator $-T_p$ is $m$-dissipative for every $p\in[1,\infty]$.    
\end{theo}
\begin{proof}
We split the proof into different steps. Throughout the proof, $\mu$ is a positive constant, arbitrarily fixed.

{\em Step 1}. Here, we prove that the operator $\mu+T$ is injective, and, if $\uu\in D(T)$ and $(\mu+T_p)\uu\in L^p(\R^d,\R^m)$, then $\uu$ belongs to $L^p(\R^d,\R^m)$ and
\begin{equation}\label{eq:stima:Tp:iniettivo}
    \mu\|\uu\|_p\leq \|(\mu+T_p)\uu\|_p.    
\end{equation}

To begin with, we observe that, from Lemma \ref{lem:kato_vett} and recalling that $\lambda_V|\xi|^2\le\langle V\xi,\xi\rangle$ for every $\xi\in\R^m$, we can easily deduce that
\begin{equation}
(\mu-\Delta+\lambda_V)\|\uu\|\leq \chi_{\{\uu\neq \bm{0}\}}\|\uu\|^{-1}\langle \f,\uu\rangle\leq \|\f\|,
\label{NSWE}
\end{equation}
in the sense of distributions, for every $\uu\in D(T)$, where $\f=(\mu+T)\uu$.

Since $\lambda_V\geq0$ almost everywhere in $\R^d$, due to Hypothesis \ref{hyp-1}, from \eqref{NSWE} it follows that
$(\mu-\Delta)\|\uu\|
\leq \|\f\|$ in the sense of distributions, i.e.,
\begin{align}
\int_{\R^d}
\|\uu\| (\mu\varphi-\Delta\varphi) dx\le\int_{\R^d}\|\f\|\varphi dx
\label{vasco}
\end{align}
for every nonnegative function $\varphi\in C^{\infty}_c(\R^d)$. Since $\|\uu\|$ and $\|\f\|$ belongs to $L^1(\R^d)+L^{\infty}(\R^d)$, by a standard truncation argument we can extend the previous formula to every nonnegative $\varphi\in {\mathcal S}(\R^d)$. Then, recalling that the operator $\mu-\Delta$ is an isomorphism in ${\mathcal S}$ which preserves the positivity, we can recast
\eqref{vasco} into the form
\begin{align}
\int_{\R^d}
\|\uu\|\varphi dx\le\int_{\R^d}\|\f\|(\mu-\Delta)^{-1}\varphi dx,
\label{vasco-1}
\end{align}
for every nonnegative $\varphi\in {\mathcal S}(\R^d)$.
From \eqref{vasco-1}, we conclude that the tempered distributions $\|\uu\|$ and $(\mu-\Delta)^{-1}\|\f\|$ satisfy the inequality
\begin{align}
\int_{\R^d}
\|\uu\|\varphi dx\le\int_{\R^d}(\mu-\Delta)^{-1}\|\f\|\varphi dx,
\label{vasco-2}
\end{align}
for every nonnegative $\varphi\in {\mathcal S}(\R^d)$.
In particular, if $\f=\bm{0}$, then $\|\uu\|=0$ almost everywhere in $\R^d$, so that $\mu+T$ is one to one.

Finally, we observe that, if $\f\in L^p(\R^d,\R^m)$ for $1\le p\le\infty$, then
the tempered distribution $(\mu-\Delta)^{-1}\|\f\|$ is actually an element of $L^p(\R^d)$. Moreover, 
$\|(\mu-\Delta)^{-1}\|_{{\mathcal L}(L^p(\R^d))}\le\mu^{-1}$, so that, from \eqref{vasco-2} it follows that 
$\|\uu\|\le (\mu-\Delta)^{-1}\|\f\|$ almost everywhere in $\R^d$ and, consequently, taking the $L^p$-norms,
$\|\uu\|_p\le\mu^{-1}\|\f\|_p$.

{\em Step 2}. Here, we prove that $T$ and $T_p$ are closed operators, and $\mu+T_p$ has closed range.

Let $(\uu_n)_{n\in\N}\subset D(T)$ be such that $\uu_n$
and $\f_n\coloneqq T\uu_n$ converge, respectively, to $\uu$ and $\f$ in $X$. From \eqref{NSWE} with $\mu=0$ we deduce that
\begin{align*}
\int_{\R^d}\lambda_V\|\uu_n-\uu_m\|\varphi dx\leq \int_{\R^d}\|\f_n-\f_m\|\varphi dx+\int_{\R^d}\|\uu_n-\uu_m\|\Delta\varphi dx,\qquad\;\,m,n\in\N,     
\end{align*}
for every nonnegative function $\varphi\in C^{\infty}_c(\R^d)$. Recalling that $\|V\xi\|\leq C\lambda_V\|\xi\|$ almost everywhere in $\R^d$ for every $\xi\in \R^m$, we infer that
\begin{align*}
0\leq & C^{-1}\int_{\R^d}\|V (  \uu_n-\uu_m)\|\varphi dx \\
\leq & \int_{\R^d}\lambda_V   \|\uu_n-\uu_m\|\varphi dx\leq \int_{\R^d}\|\f_n-\f_m\|\varphi dx+\int_{\R^d}\|\uu_n-\uu_m\|\Delta \varphi dx.  
\end{align*}
This implies that $(V\uu_n\varphi)_{n\in\N}$ is a Cauchy sequence in $L^1(\R^d,\R^m)$. Moreover, since
$X$ is continuously embedded in $L^1(\Omega,\R^m)$ for
every bounded measurable set $\Omega$, up to a subsequence, we can assume that $(\uu_n)_{n\in\N}$ pointwise converges to $\uu$ almost everywhere in $\R^d$. As a consequence, $(V\uu_n)_{n\in\N}$ pointwise converges to $V\uu$, so that $(V\uu_n\varphi)_{n\in\N}$ converges to $V\uu\varphi$ in $L^1(\R^d,\R^m)$. The arbitrariness of $\varphi\in C^\infty_c(\R^d)$ implies that $V\uu\in L^1_{\rm loc}(\R^d,\R^m)$ and so $(V\uu_n)_{n\in\N}$ converges to $V\uu$ in $L^1_{\rm loc}(\R^d,\R^m)$.
Hence, $(-\Delta (\uu_n)_j=(\f_n)_j - (V\uu_n)_j)_{n\in\N}$ converges to $f_j - (V\uu)_j$ in $\calD'$ for every $j=1,\ldots,m$, which means that $-\Delta\uu + V\uu=\f\in X$, i.e., $\uu\in D(T)$ and $T\uu=\f$.

The same arguments can be applied to show that also the operator $T_p$ is closed in $L^p(\R^d,\R^m)$ for every $p\in [1,\infty]$.

Finally, the closedness of the range of $\mu+T_p$ is a consequence of inequality \eqref{eq:stima:Tp:iniettivo} in Step 1 and the closedness of the operator $T_p$.

{\em Step 3}. Here, we prove that the operators $\mu+T$ and $\mu+T_p$ are invertible and $\|(\mu+T_p)^{-1}\uu\|_p\leq \mu^{-1}\|\uu\|_p$ for every $\uu\in L^p(\R^d,\R^m)$. 

The injectivity of the operators $\mu+T$ and $\mu+T_p$ follow from Step 1. Hence, we just need to prove that the operators $\mu+T$ and $\mu+T_p$ are surjective. The main steps to prove this property are the cases $p=1$ and $p=\infty$. Once the surjectivity of the operators $\mu+T_1$ and $\mu+T_{\infty}$ is proved, the other cases will follow almost straightforwardly.

From Step 2, we know that $\mu+T_1$ has closed range. Hence, it remains to prove that $\mu+T_1$ has dense range. Assume that 
\begin{equation*}
\int_{\R^d}\langle \bm{v},(\mu+T_1)\uu\rangle dx=0, \qquad \uu\in D(T_1),
\end{equation*}
for some $\bm{v}\in L^\infty(\R^d,\R^m)$. Since $C^{\infty}_c(\R^d,\R^m)$ is contained in $D(T_1)$ and
$V\bm{v}\in L^1_{\rm loc}(\R^d,\R^m)$,
from the previous formula, we infer that $(\mu-\Delta+V)\bm{v}=\bm{0}$ in the sense of distributions. This implies that $\bm{v}\in D(T)$ and $(\mu+T)\bm{v}=\bm{0}$. The injectivity of $\mu+T$ yields $\bm{v}=\bm{0}$. We have so proved that the operator $T_1$ is invertible.

We now address the case $p=\infty$. For this purpose, we fix $\f\in L^\infty(\R^d,\R^m)$, with $\f\geq \bm{0}$ componentwise, and choose a  sequence $(\f_n)_{n\in\N}\subset L^\infty_c(\R^d,\R^m)$ such that $\f_n\geq \bm{0}$ and $\f_n$ pointwise increases to $\f$ almost everywhere in $\R^d$. Clearly, each function $\f_n$ belongs to $L^2(\R^d,\R^m)$. 
Since the operator ${\mathcal H}:D({\mathcal H})\to L^2(\R^d,\R^m)$ is accretive, the operator $(\mu+\mathcal{H})^{-1}$ is well-defined and bounded in $L^2(\R^d,\R^m)$. For every $n\in\N$, let us set
$\uu_n=(\mu+{\mathcal H})^{-1}\f_n$. Note that each function $\uu_n$ belongs to $X$. Moreover, $V\uu_n\in L^1_{\rm loc}(\R^d,\R^m)$ as a consequence of \eqref{stima-L1-V-eps-u-eps}, which shows that $(\mu+\lambda_V)\|\uu_n\|$ belongs to $L^1(\R^d)$.
Finally, by difference, $-\Delta\uu_n+V\uu_n=\f_n-\mu\uu_n$ belongs to $X$. We have so proved that $\uu_n\in D(T)$. Finally, observing that $(\mu+T)\uu_n=\f_n$ in the sense of distributions and $\f_n\in L^{\infty}_c(\R^d,\R^m)$, for every $n\in\N$, 
from Step 1, we conclude that $\uu_n\in L^{\infty}(\R^d,\R^m)$ and $\|\uu_n\|_\infty\leq \mu^{-1}\|\f_n\|_\infty\leq \mu^{-1}\|\f\|_\infty$ for every $n\in\N$. It thus follows that 
$\uu_n\in D(T_\infty)$ and $(\mu+T_\infty)\uu_n=\f_n$ for every $n\in\N$.
Since $\f_n\geq\bm{0}$ and increases componentwise, from Proposition \ref{prop:conv_u_eps_M} we infer that $\uu_n\geq\bm{0}$ and increases componentwise. Hence, there exists the pointwise limit $\uu\coloneqq\lim_{n\to\infty}\uu_n\in L^\infty(\R^d,\R^m)$. 

Let us show that $\uu$ belongs to $D(T_{\infty})$  and $(\mu+T_\infty)\uu=\f$. Since $\|\uu_n\|_\infty\leq\mu^{-1}\|\f\|_\infty$ for every $n\in\N$, it follows that $(V\uu_n)_{n\in\N}$ converges to $V\uu$ in $L^1_{\rm loc}(\R^d,\R^m)$ by dominated convergence. Moreover, $\|\uu\|_{\infty}\le \mu^{-1}\|\f\|_{\infty}$ and, by difference, 
\begin{align*}
\int_{\R^d}\langle\f,\bm{\varphi}\rangle dx=&
\lim_{n\to +\infty}\int_{\R^d}\langle\f_n,\bm{\varphi}\rangle dx
=\lim_{n\to +\infty}\int_{\R^d}\langle \uu_n,(\mu-\Delta+V)\bm{\varphi}\rangle dx\\
=&\mu\int_{\R^d}\langle\uu,\bm{\varphi}\rangle dx
-\int_{\R^d}\langle\uu,\Delta\bm{\varphi}\rangle dx
+\int_{\R^d}\langle V\uu,\bm{\varphi}\rangle dx, \qquad \bm{\varphi}\in C^\infty_c(\R^d,\R^m).
\end{align*}
This shows that the distributional Laplacian of $\uu$ coincides with the function $\mu\uu+V\uu-\f$. In particular, $-\Delta\uu+V\uu\in L^{\infty}(\R^d,\R^m)$, so that $\uu\in D(T_{\infty})$ and $(\mu+T_{\infty})\uu=\f$.
We have thus proved that $T_{\infty}$ is invertible.

Since $\mu+T_1$ and $\mu+T_{\infty}$ are invertible, we can now easily show that also the operators $T$ and $T_p$ ($p\in (1,\infty)$) are invertible. We begin by considering the operator $T$. Fix $\f\in X$ and let $\f_1\in L^1(\R^d,\R^m)$ and $\f_{\infty}\in L^\infty(\R^d,\R^m)$ be such that $\f=\f_1+\f_{\infty}$. From the above results, there exist $\uu_1\in D(T_1)$ and $\uu_{\infty}\in D(T_\infty)$ such that $(\mu+T_1)\uu_1=\f_1$ and $(\mu+T_\infty)\uu_{\infty}=\f_{\infty}$. Clearly, the function $\uu=\uu_1+\uu_{\infty}$ belongs to $X$ and, recalling that $T_1$ and $T_{\infty}$ are the parts of $T$ in $L^1(\R^d,\R^m)$ and in $L^{\infty}(\R^d,\R^m)$, respectively, we conclude that $(\mu+T)\uu=(\mu+T)(\uu_1+\uu_{\infty})=(\mu+T_1)\uu_1+(\mu+T_{\infty})\uu_{\infty}=\f_1+\f_{\infty}=\f$.

Finally, we show that $\mu+T_p$ is invertible for every $p\in(1,\infty)$. Fix $\f\in L^p(\R^d,\R^m)\subset X$ and let $\uu\in D(T)$ be such that $(\mu+T)\uu=\f\in L^p(\R^d,\R^m)$. By Step 1, we deduce that $\uu\in L^p(\R^d,\R^m)$. Hence, $\uu$ belongs to $D(T_p)$ and $(\mu+T_p)\uu=\f$. The estimate $\|(\mu+T_p)^{-1}\uu\|_p\leq \mu^{-1}\|\uu\|_p$ follows again from\eqref{eq:stima:Tp:iniettivo}. The proof is complete.
\end{proof}

\begin{theo}
Under Hypotheses $\ref{hyp-1}$, $\ref{hyp-2}$ and assuming that $\lambda_V$ belongs to $B_p$ for some $p\in [1,\infty)$, the realization $-A_p$ of the operator $\Delta -V$ in $L^p(\R^d,\R^m)$, with $D(A_p)=\{\uu\in L^p(\R^d,\R^m):
\Delta\uu, V\uu\in L^p(\R^d,\R^m)\}$, generates a strongly continuous semigroup of contraction, which is, for $p>1$, also analytic.
Moreover, $D(A_p)$ coincides with the Kato maximal domain, i.e.
\begin{align*}
D(A_p)=\{\uu\in L^p(\R^d,\R^m): \Delta\uu-V\uu\in L^p(\R^d,\R^m), V\uu\in L^1_{\rm loc}(\R^d,\R^m)\}    
\end{align*}
and, if $p\ge 2$, then
\begin{align*}
D(A_p)=\{\uu\in L^p(\R^d,\R^m): \Delta\uu-V\uu\in L^p(\R^d,\R^m)\}.    
\end{align*}
\end{theo}

\begin{proof}
We fix $p\in [1,\infty)$ and split the proof into three steps. 
In the first one we prove that $C^{\infty}_c(\R^d,\R^m)$ is a core for the operator $T_p$ and, in the second one, we prove that $D(T_p)=D(A_p)$. Finally, in Step 3, we prove that $A_p$ is a sectorial operator.

{\em Step 1}. To prove that $C^{\infty}_c(\R^d,\R^m)$ is core for the operator $T_p$, we just need to show that 
$C^{\infty}_c(\R^d,\R^m)$ is contained in $D(T_p)$ and 
$(1+T_p)(C^\infty_c(\R^d,\R^m))$ is dense in $L^p(\R^d,\R^m)$. The first property is immediate to prove.
To prove the second property, we fix  $\bm{v}\in L^{p'}(\R^d,\R^m)$ such that 
\begin{align*}
0=\int_{\R^d}\langle (1+T_p)\uu,\bm{v}\rangle dx=
\sum_{j=1}^m\int_{\R^d}(u_j-\Delta u_j+(Vu)_j)v_jdx
, \qquad\;\, \forall \uu\in C^\infty_c(\R^d,\R^m)
\end{align*}
or, equivalently,
\begin{align*}
\sum_{j=1}^m\int_{\R^d}\Delta u_j v_jdx=
\sum_{j=1}^m\int_{\R^d}u_j(Vv)_jdx+
\sum_{j=1}^m\int_{\R^d}u_jv_jdx, \quad \forall \uu\in C^\infty_c(\R^d,\R^m).
\end{align*}
Taking $\uu=\varphi \bm{e}_k$ $(k=1,\ldots,m)$, where $\bm{e}_k$ is the $k$-th element of the canonical basis of $\R^m$, $\varphi\in C^{\infty}_c(\R^d)$, and recalling that $\lambda_V\in B_p$, we deduce that the distributional Laplacian of $v_k$ belongs to $L^1_{\rm loc}(\R^d)$ (since $(Vv)_k\in L^1_{\rm loc}(\R^d)$) and coincides with the function
$(Vv)_k+v_k$. This shows that $\Delta\bm{v}-V\bm{v}=\bm{v}\in L^{p'}(\R^d,\R^m)$, so that
$\bm{v}\in D(T_{p'})$ and $(1+T_{p'})\bm{v}=\bm{0}$.
The injectivity of $(1+T_{p'})$ implies that $\bm{v}=\bm{0}$. Thus, we have proved that $(1+T_p)(C^{\infty}_c(\R^d,\R^m))$ is dense in $L^p(\R^d,\R^m)$.

{\em Step 2}. By Theorem \ref{the:generation}, the operator $T_p$ generates a strongly continuous semigroup of contractions in $L^p(\R^d,\R^m)$. Clearly, 
$D(A_p)\subset D(T_p)$. To prove that actually the two sets coincide, we fix $\uu\in D(T_p)$ and a sequence $(\uu_n)_{n\in\N}\subset C^{\infty}_c(\R^d,\R^m)$ such that
$\uu_n$ and $-\Delta\uu_n+V\uu_n$ converge, respectively, to
$\uu$ and $T_p\uu$ as $n$ tends to $\infty$.
Estimate \eqref{coltellate}, if $p=1$, and estimate
\eqref{eq:maximal_ineq:Lp}, if $p\in (1,\infty)$, show that $(\Delta\uu_n)_{n\in\N}$ and $(V\uu_n)_{n\in\N}$ are Cauchy sequences in $L^p(\R^d,\R^m)$, so that they converge, respectively, to some functions $\bm{w}_1$ and $\bm{w}_2$. Since $\uu_n$ converges to $\uu$ in $L^p(\R^d,\R^m)$, it is immediate to infer that $\bm{w}_2=V\uu$, so that $V\uu$ belongs to $L^p(\R^d,\R^m)$. By difference also the Laplacian of $\uu$ belongs to $L^p(\R^d,\R^m)$. We have so proved that $\uu\in D(A_p)$. Hence, $D(A_p)$ is the Kato maximal domain
of the realization of $-\Delta+V$ in $L^p(\R^d,\R^m)$.
Finally, we observe that, if $p\ge 2$, then $L^p(\R^d,\R^m)$ is contained in $L^{p'}_{\rm loc}(\R^d,\R^m)$, so that H\"older's inequality implies that
$V\uu\in L^1_{\rm loc}(\R^d,\R^m)$ for any $\uu \in L^p(\R^d,\R^m)$, and $D(A_p)$ actually coincides with the maximal domain of the realization of $-\Delta+V$ in $L^p(\R^d,\R^m)$.

{\em Step 3}. 
To prove the sectoriality of $A_p$, 
in view of \cite[Chapter I, Section 5.8]{goldstein} and taking into account that $C^{\infty}_c(\R^d,\C^m)$ is a core of $A_p$, it suffices to show that there exists a positive constant $C_p$ such that 
\begin{equation}
\left|{\rm Im} \int_{\R^d}(A_p\uu,\uu)\norm{\uu}^{p-2} dx\right| \leq  -C_p{\rm Re} \int_{\R^d}(A_p\uu,\uu)\norm{\uu}^{p-2} dx,\qquad\;\,
\bm{u}\in C_c^\infty(\R^d, \C^m).
\label{regular-diss}
\end{equation}
Here with $(\cdot,\cdot)$ we indicate the Hermitian scalar product of $\C^m$. 
By \cite[Theorem 3.9]{Ouhabaz} it follows that 
\begin{equation}\label{Maati}
-{\rm Re}\int_{\R^d}(\Delta \uu,\uu)\norm{\uu}^{p-2}dx\ge c_p\left|{\rm Im}\int_{\R^d}(\Delta \uu,\uu)\norm{\uu}^{p-2}dx\right|,\qquad\;\, \uu\in C_c^\infty(\R^d,\C^m),
\end{equation}
where $c_p$ is $\frac{2\sqrt{p-1}}{|p-2|}$ if $p\neq 2$ and any positive constant if $p=2$.
Recalling that $(V\uu,\uu)\ge 0$, \eqref{regular-diss} follows from \eqref{Maati}, where $C_p=c_p^{-1}$. For more details, see the proof of \cite[Proposition 4.5]{MR18}.
\end{proof}

\section{Examples}
\label{sec:Exa}
We conclude this work by provide two matrix-valued potentials $V$ which satisfy the assumptions of Theorem \ref{thm:maximal_ineq}. In particular, we show that entries with singularities for $V$ are allowed. We have to prove that Hypotheses \ref{hyp-1} and \ref{hyp-2} are satisfied, and that the minimal eigenvalue of $V$ belongs to a $B_q$ class, for some $q\in(1,\infty)$.

\begin{exa}\label{exemple1}
Let $m=2$ and $1<q<\infty$. Then, for every $x\in\R^d\setminus\{0\}$ we consider the symmetric matrix $V=(v_{ij}(x))_{i,j=1}^2$ with entries
\begin{align*}
v_{11}(x) &=\norm{x}^{-\alpha}((\cos(\|x\|))^2 + c_1(\sin(\|x\|))^2+ \norm{x}^{\beta}(k(\cos(\|x\|))^2 + k_1(\sin(\|x\|))^2),\\
v_{12}(x) = v_{21}(x)&=|\sin(\|x\|)\cos(\|x\|)| \left[(1 - c_1)\norm{x}^{-\alpha}+(k-k_1)\norm{x}^{\beta}\right],\\ 
v_{22}(x) &=\norm{x}^{-\alpha}(c_1(\cos(\|x\|))^2+ (\sin(\|x\|))^2) + \norm{x}^{\beta}(k_1(\cos(\|x\|))^2 + k(\sin(\|x\|))^2),
\end{align*}
for every $x\in\R^d\setminus\{0\}$, where $\beta$, $c_1$, $k$ and $k_1$ are positive constants, with $c_1>1$ and $k_1>k$, and $\alpha\in \left (0,\frac{d}{q}\right )$.

Clearly, the off-diagonal terms are non-positive.
Moreover, since 
\begin{align*}
\operatorname{det}(\lambda I-V(x)) =& (\lambda-v_{11})(\lambda-v_{22}(x)) - (v_{12}(x))^2\\
=& \lambda^2 - \lambda \left[\norm{x}^{-\alpha}(1+c_1)+\norm{x}^{\beta}(k_1+k)\right]\\
&+(c_1\norm{x}^{-2\alpha}+k\,k_1\norm{x}^{2\beta})\left((\cos(\|x\|))^4+(\sin(\|x\|))^4+2(\sin(\|x\|))^2(\cos(\|x\|))^2\right)\\
&+\norm{x}^{\beta-\alpha}\big (k_1(\cos(\|x\|))^4 + c_1k(\sin(\|x\|))^4 + c_1k(\cos(\|x\|))^4+ k_1(\sin(\|x\|))^4\big )\\
&+\norm{x}^{\beta-\alpha}\left(2k_1(\sin(\|x\|))^2(\cos(\|x\|))^2 + 2c_1k(\sin(\|x\|))^2(\cos(\|x\|))^2\right)\\
=& \lambda^2 - \lambda \left[\norm{x}^{-\alpha}(1+c_1)+\norm{x}^{\beta}(k_1+k)\right]\\ 
&+c_1\norm{x}^{-2\alpha} + kk_1 \norm{x}^{2\beta} + (k_1+kc_1)\norm{x}^{\beta-\alpha}\\
=&\left(\lambda-\|x\|^{-\alpha}-k\|x\|^{\beta}\right)\left(\lambda-c_1\|x\|^{-\alpha}-k_1\|x\|^{\beta}\right),
\end{align*}
the eigenvalues of $V(x)$ are
\begin{align*}
\lambda_1(x)\coloneqq\|x\|^{-\alpha}+k\|x\|^{\beta},\quad\lambda_2(x)\coloneqq c_1\|x\|^{-\alpha}+k_1\|x\|^{\beta}.
\end{align*}
They are positive functions, and therefore $V$ is a positive defined matrix-valued function. Moreover,   
\begin{align*}
\lambda_1(x)\le\lambda_2(x)\le C\lambda_1(x),\qquad\;\,x\in\R^d\setminus\{0\}, 
\end{align*}
where $C\coloneqq\max\left\{c_1,k_1/k\right\}$.
Hence, Hypothesis \ref{hyp-1} and \eqref{hyp-2} are satisfied.
		
In addition, the condition $0<\alpha<\frac{d}{q}$ and  $\beta>0$ ensures that $\lambda_V=\lambda_1$ belongs to $B_q$ since the function $x\mapsto\|x\|^{-\gamma}$ belongs to  $B_q$ for $\gamma\in \left (-\infty,\frac{d}{q}\right )$ (see for instance \cite[Chapter 9]{grafakos-modern}).
\end{exa}

\begin{exa}\label{exemple2}
Let $\widetilde V:\R^d\to \R^{m^2}$ be the matrix-valued matrix,  whose entries are measurable functions and
there exist positive constants $c_i<C_i$, $C_{ij}$ and ${\eta}_{ij}<{\eta}$ with $C_{ij}=C_{ji}$ and ${\eta}_{ij}={\eta}_{ji}$ for every $i,j=1,\ldots,m$, such that
\begin{align*}
c_i(1+\|x\|^2)^{\eta}\leq \widetilde v_{ii}(x)\leq C_{i}(1+\|x\|^2)^{\eta}, \qquad -C_{ij}(1+\|x\|^2)^{{\eta}_{ij}}\leq \widetilde v_{ij}(x)=\widetilde v_{ji}(x)\leq 0    
\end{align*}
for every $x\in\R^d$ and every $i=1,\ldots,m$ and $j\in\{1,\ldots,m\}\setminus\{i\}$. Further, we  assume that
\begin{align*}
\sum_{i=1}^m\bigg(c_i\xi_i^2-\sum_{j=1, j\neq i}^m C_{ij}\xi_i\xi_j\bigg)\geq 0,\qquad\;\,\xi\in\R^m.   
\end{align*}
This implies that $\widetilde V(x)$ is positive semi-definite for every $x\in\R^d$. Indeed, for every $x\in\R^d$ and $\xi\in \R^m$ we get
\begin{align*}
\langle \widetilde V(x)\xi,\xi\rangle
\geq & \sum_{i=1}^m\bigg(c_i(1+\|x\|^2)^{\eta}\xi_i^2-\sum_{j=1,j\neq i}^m C_{ij}(1+\|x\|^2)^{{\eta}_{ij}}|\xi_i||\xi_j|\bigg) \\
= & (1+\|x\|^2)^{\eta}\sum_{i=1}^m\bigg(c_i\xi_i^2-\sum_{j=1,j\neq i}^m C_{ij}(1+\|x\|^2)^{{\eta}_{ij}-{\eta}}|\xi_i||\xi_j|\bigg) \\
\geq & (1+\|x\|^2)^{\eta}\sum_{i=1}^m\bigg(c_i\xi_i^2-\sum_{j=1,j\neq i}^m C_{ij}|\xi_i||\xi_j|\bigg)\geq 0.
\end{align*}
Moreover, if $x\in\R^d\setminus B(0,1)$, arguing as above, we deduce that
\begin{align}
\langle \widetilde V(x)\xi,\xi\rangle
\geq & (1+\|x\|^2)^{\eta}\sum_{i=1}^m\bigg(c_i\xi_i^2-2^{\max\{{\eta}_{ij}:i,j=1,\ldots,m\}-{\eta}}\sum_{j=1,j\neq i}^m C_{ij}|\xi_i||\xi_j|\bigg)\notag\\
\geq & (1-2^{\max\{{\eta}_{ij}:i,j=1,\ldots,m\}-{\eta}})(1+\|x\|^2)^{\eta}\sum_{i=1}^mc_i\xi_i^2\notag\\
\geq & \min\{c_i:i=1,\ldots,m\} (1-2^{\max\{{\eta}_{ij}:i,j=1,\ldots,m\}-{\eta}})(1+\|x\|^2)^{\eta}\|\xi\|^2.
\label{europa}
\end{align}

Finally,
\begin{align}
\langle\widetilde V(x)\xi,\xi\rangle\le 
(1+\|x\|^2)^{\eta}\sum_{i=1}^m\bigg(c_i\xi_i^2-\sum_{j=1,j\neq i}^m C_{ij}|\xi_i||\xi_j|\bigg)
\label{league}
\end{align}
for every $x\in\R^d$ and $\xi\in\R^m$.

From \eqref{europa} and \eqref{league} it follows that there exist positive constants $\widetilde D_m$ and $\widetilde D_M$ such that
\begin{align}
\label{ex_2_avl_v_tilde}
\widetilde D_m(1+\|x\|^2)^{\eta}\leq
\lambda_{\widetilde V}(x)\leq \Lambda_{\widetilde V}(x)\leq \widetilde D_M(1+\|x\|^2)^{\eta},\qquad\;\,
x\in\R^d\setminus B(0,1),
\end{align}
where $\lambda_{\widetilde V}(x)$ and $\Lambda_{\widetilde V}(x)$ denote, respectively, the minimum and the maximum eigenvalue of the matrix $\widetilde V(x)$.

Let us set $V(x)=\widetilde V(x)+\|x\|^{-\alpha}{\rm Id}$ for every $x\in\R^d\setminus\{0\}$ and some $\alpha\in \left(0,\frac dq\right)$.
Then, the matrix-valued function $V$, arbitrarily extended to the whole $\R^d$ in such a way that the matrix $V(0)$ is definite positive,
satisfies Hypothesis \ref{hyp-1}.

Moreover, since for every $x\in B(0,1)$ it holds that
\begin{align*}
\|x\|^{-\alpha}
\geq & \frac{1}{2}\|x\|^{-\alpha}+\frac{1}{2}
\geq \frac12\|x\|^{-\alpha}+2^{-1-\eta}(1+\|x\|^2)^\eta,
\end{align*}
it follows that there exist positive constants $D_m$ and $D_M$ such that 
\begin{align}
D_m\left(\|x\|^{-\alpha}+(1+\|x\|^2)^{\eta}\right)\leq
\lambda_V(x)\leq \Lambda_V(x)\leq D_M\left(\|x\|^{-\alpha}+(1+\|x\|^2)^{\eta}\right),    
\label{braga}
\end{align}
for every $x\in B(0,1)\setminus\{0\}$.
Using \eqref{ex_2_avl_v_tilde}, we can easily extend \eqref{braga} to every $x\in\R^d\setminus\{0\}$, up to replacing $D_m$ and $D_M$ with suitable new positive constants.

Hence, Hypothesis \ref{hyp-2} is fulfilled with $C=D_M/D_m$ and, since the function $x\mapsto\|x\|^{-\alpha}+(1+\|x\|^2)^{\eta}$ belongs to $B_q$, it follows that $\lambda_V\in B_q$.
\end{exa}

\end{document}